\documentclass[a4paper, oneside, 11pt]{amsart}

\usepackage[utf8]{inputenc}
\usepackage[T1]{fontenc}
\usepackage{microtype}
\usepackage[english]{babel}
\usepackage{lmodern}
\usepackage{geometry}

\usepackage{xparse}
\usepackage{bbm}

\numberwithin{equation}{section}

\usepackage{lastpage}
\usepackage{setspace}
\usepackage[shortlabels]{enumitem}
\usepackage{romannum}
\usepackage[dvipsnames]{xcolor}
\usepackage{csquotes}

\usepackage[colorlinks, linkcolor=Blue,citecolor=Blue, urlcolor=Blue, final]{hyperref}

\usepackage[abbrev,backrefs]{amsrefs}

\usepackage{amssymb}
\usepackage{mathrsfs}
\usepackage{amsthm}

\usepackage{tikz}

\usetikzlibrary{intersections,arrows.meta,math,calc,decorations.markings}
\usetikzlibrary{backgrounds}
\tikzset{middlearrow/.style={
        decoration={markings,
            mark= at position 0.5 with {\arrow{#1}} ,
        },
        postaction={decorate}
    }
}

\usepackage{thmtools}
\usepackage[capitalise, nameinlink]{cleveref}

\usepackage{mathtools}
\DeclarePairedDelimiterX{\inpr}[2]{(}{)}{#1\delimsize \vert #2}
\DeclarePairedDelimiterX{\intvc}[2]{[}{]}{#1,#2}
\DeclarePairedDelimiterX{\intvl}[2]{(}{]}{#1,#2}
\DeclarePairedDelimiterX{\intvr}[2]{[}{)}{#1,#2}
\DeclarePairedDelimiterX{\intvo}[2]{(}{)}{#1,#2}
\usepackage[mathscr]{euscript}

\newcommand*{\Diff}{\mathrm D} %Différentielle
\newcommand*{\dd}{\mathop{}\!{\operatorfont d}} %d du dx et de la dérivée
\newcommand*{\R}{\mathbb{R}} %Reel
\newcommand*{\C}{\mathbb{C}} %Complexes
\newcommand*{\SP}{\mathbb{S}} %Sphère
\newcommand*{\e}{\mathrm{e}} %Nombre d'Euler
\newcommand*{\ii}{\mathrm{i}} %Unité imaginaire

\newcommand{\id}{\mathrm{id}} %Opérateur identité
\DeclareMathOperator{\Supp}{supp} %Support d'une fonction
\DeclareMathOperator{\Tr}{Tr} %Trace
\DeclareMathOperator{\Ext}{Ext} %Extension
\usepackage{stackrel}

\makeatletter
\newsavebox{\@brx}
\newcommand{\llangle}[1][]{\savebox{\@brx}{\(\m@th{#1\langle}\)}%
  \mathopen{\copy\@brx\kern-0.5\wd\@brx\usebox{\@brx}}}
\newcommand{\rrangle}[1][]{\savebox{\@brx}{\(\m@th{#1\rangle}\)}%
  \mathclose{\copy\@brx\kern-0.5\wd\@brx\usebox{\@brx}}}
\makeatother

\DeclareMathOperator{\Pt}{Pt}

\DeclarePairedDelimiter\brk{(}{)}

\DeclarePairedDelimiter\sqb{[}{]}

\DeclarePairedDelimiterX\norm[1]\lvert\rvert{ %norme (simple barre)
\ifblank{#1}{\,\cdot\,}{#1}
}

\DeclarePairedDelimiterX\Norm[1]\lVert\rVert{ %Norme (double barre)
\ifblank{#1}{\,\cdot\,}{#1}
}

\DeclarePairedDelimiterX{\floor}[1]{\lfloor}{\rfloor}{ %Partie entière
\ifblank{#1}{\,\cdot\,}{#1}
}

\DeclarePairedDelimiterX{\ceil}[1]{\lceil}{\rceil}{    %Partie entière supérieure
\ifblank{#1}{\,\cdot\,}{#1}
}

\DeclarePairedDelimiterXPP\pNorm[1]{}{\lVert}{\rVert}{_p}{   %Norme Lp
\ifblank{#1}{\,\cdot\,}{#1}
}

\DeclarePairedDelimiterXPP\spNorm[2]{}{\lVert}{\rVert}{_{#2}}{   %Norme d'un espace 
\ifblank{#1}{\,\cdot\,}{#1}
}

\DeclarePairedDelimiterXPP\Lin[2]{\mathrm{Lin}}{(}{)}{}{#1, #2} 

\providecommand{\st}{\,\vert\,}
\newcommand\stSymbol[1][]{%
\nonscript\;#1\vert
\allowbreak
\nonscript\;
\mathopen{}}
\DeclarePairedDelimiterX\set[1]\{\}{%
\renewcommand\st{\stSymbol[\delimsize]}
#1
}
\theoremstyle{plain}
\declaretheorem[name=Theorem, numberwithin=section]{theorem}
\declaretheorem[name=Theorem, numbered=no]{theorem*}
\declaretheorem[name=Definition, sibling=theorem]{definition}
\declaretheorem[name=Definition, numbered=no]{definition*}
\declaretheorem[name=Lemma, sibling=theorem]{lemma}
\declaretheorem[name=Lemma,numbered=no]{lemma*}
\declaretheorem[name=Proposition, sibling=theorem]{proposition}
\declaretheorem[name=Proposition, numbered=no]{proposition*}

\declaretheorem[name=Corollary, numbered=no]{corollary*}

\declaretheorem[name=Affirmation,numbered=no]{affirmation*}
\theoremstyle{definition}

\declaretheorem[name=Exercise,numbered=no, qed={$\triangle$}]{exo*}

\declaretheorem[name=Example,numbered=no, qed={$\triangle$}]{example*}
\declaretheorem[name=Remark, sibling=theorem,qed={$\triangle$}]{rmq}
\declaretheorem[name=Remark,numbered=no, qed={$\triangle$}]{rmq*}

\renewcommand{\PrintDOI}[1]{%
  \href{https://doi.org/#1}{doi:#1}%
}
\BibSpec{book}{%
    +{}  {\PrintPrimary}                {transition}
    +{,} { \textit}                     {title}
    +{.} { }                            {part}
    +{:} { \textit}                     {subtitle}
    +{,} { \PrintEdition}               {edition}
    +{}  { \PrintEditorsB}              {editor}
    +{,} { \PrintTranslatorsC}          {translator}
    +{,} { \PrintContributions}         {contribution}
    +{,} { }                            {series}
    +{,} { \voltext}                    {volume}
    +{,} { }                            {publisher}
    +{,} { }                            {organization}
    +{,} { }                            {address} 
    +{,} { \PrintDateB}                 {date}
    +{,} { }                            {status}
    +{,} { \PrintDOI}                   {doi}
    +{}  { \parenthesize}               {language}
    +{}  { \PrintTranslation}           {translation}
    +{;} { \PrintReprint}               {reprint}
    +{.} { }                            {note}
    +{.} {}                             {transition}
    +{}  {\SentenceSpace \PrintReviews} {review}
}

\title{Trace theory for gauge-covariant Sobolev spaces}
\author[J.\ Van Schaftingen]{Jean Van Schaftingen}
\address[J. Van Schaftingen]{
Universit\'e catholique de Louvain, Institut de Recherche en Math\'ematique et Physique, Chemin du Cyclotron 2 bte L7.01.01, 1348 Louvain-la-Neuve, Belgium}
\email{Jean.VanSchaftingen@uclouvain.be}
\address[L. Winter]{
Universit\'e catholique de Louvain, Institut de Recherche en Math\'ematique et Physique, Chemin du Cyclotron 2 bte L7.01.01, 1348 Louvain-la-Neuve, Belgium}
\email{Leon.Winter@uclouvain.be}
\author[L.\ Winter]{Leon Winter}
\date{\today}

\thanks{Both authors were supported by the Fonds Spéciaux de Recherche (FSR), UClouvain. Van Schaftingen was supported by the Projet de Recherche T.0229.21 ``Singular Harmonic Maps and Asymptotics of Ginzburg--Landau Relaxations'' of the Fonds de la Recherche Scientifique--FNRS}


\begin{document}
\pagenumbering{arabic}

\begin{abstract}
The traces of gauge-covariant Sobolev spaces on a Riemannian vector bundle for some connection are characterised as some gauge-covariant fractional Sobolev spaces when the curvature of the connection is bounded. 
The constants in the trace and extension theorems only depend on this curvature.
When the connection is abelian, one recovers known results for magnetic Sobolev spaces.
\end{abstract}

\maketitle

\section{Introduction}
A \emph{gauge-covariant Sobolev space} \(W^{1,p}_K(\mathcal{M},\mathcal{E})\) on a smooth Riemannian vector bundle \(\mathcal E\) above a smooth Riemannian manifold \(\mathcal{M}\) endowed with a metric connection \(K\) is a set of sections \(U \colon \mathcal{M} \to \mathcal E\) of \(\mathcal E\) for which a \emph{weak covariant derivative} \(\Diff_K U\) satisfies a condition of integrability. When the connection \(K\) is trivial, the gauge-covariant Sobolev spaces coincide with their vector valued counterparts, the usual Sobolev spaces, thereby generalising them.
The space \(W^{1, 2}_K(\mathcal{M},\mathcal{E})\) is the natural framework for scalar fields in various Yang--Mills equations~\cite{Choquet-Bruhat_Christodoulou_1981}, including  first-order Yang--Mills equations \cite{Krieger_Sterbenz_2013}, coupled boson Yang--Mills energies~\cite{Parker_1982}*{\S 2} and Yang--Mills--Higgs energies for sections in a (non-linear) bundle \citelist{\cite{Zhang_2004}\cite{Jost_Kessler_Wu_Zhu_2022}\cite{Chen_Song_2021}\cite{Pigati_Stern_2021}\cite{Song_2016}\cite{Ai_Song_Zhu_2019}\cite{Wei_Yang_Yu_2024}}.

In the study of boundary value problems in the calculus of variations and partial differential equations it is natural to determine whether a reasonable meaning can be given to \emph{boundary value} of a Sobolev-type function, known as the \emph{trace} of the function, and, if so, what are all \emph{admissible} boundary values and the associated estimates. 
For the vector valued Sobolev space 
\[
  W^{1,p}(\Omega, \R^m)
  = 
  \set[\bigg]{u \in L^p(\Omega, \R^m) \st 
  u \text{ is weakly differentiable and } \int_{\Omega} \norm{\Diff u}^p < +\infty}.
\]
Gagliardo's seminal \emph{trace theorem} \cite{Gagliardo_1957} (see also~\citelist{\cite{Uspenskii_1961}\cite{Adams_Fournier_2003}*{Thm.\ 7.39}\cite{diBenedetto_2016}*{Prop.\  17.1}\cite{Mironescu_Russ_2015}\cite{Mazya_2011}*{\S  10.1.1 Thm.\ 1}\cite{Leoni_fractionnal_2023}*{Thm.\ 9.4}}) states that if \(1 < p < +\infty\) and if \(\Omega \subset \R^{n+1}\) is a bounded open domain with Lipschitz boundary, or if \(\Omega = \R^{n+1}_+ = \R^{n}\times (0,+\infty) \subset \R^{n+1}\), then there exists a unique linear, continuous and surjective operator, called the \emph{trace operator}, 
\begin{equation*}
	\Tr \colon W^{1,p}(\Omega, \R^m) \to W^{1-1/p,p}(\partial\Omega,\R^m)
\end{equation*} 
and a linear and continuous operator, known as the \emph{extension operator}, 
\begin{equation*}
	\Ext \colon W^{1-1/p,p}(\partial \Omega, \R^m) \to W^{1,p}(\Omega, \R^m)
\end{equation*} 
such that \(\Tr \circ \Ext\) is the identity on the fractional Sobolev space 
\[
  W^{1-1/p,p} \brk{\partial\Omega, \R^m}
  =\set[\bigg]{ u \colon \partial\Omega \to \R^m \st \ \smashoperator{\iint_{\partial \Omega \times \partial \Omega}} \frac{\norm{u \brk{x} - u \brk{y}}^p}{\norm{x - y}^{p - 1 + \dim \partial \Omega}} \dd x \dd y < +\infty}.
\]
 
The goal of the present work is to extend the trace theory to Riemannian vector bundles. 
Let us first consider a trivial vector bundle \(E \cong \smash{\overline{\R}}^{n+1} \times F\) above the half space \(\smash{\overline{\R}}^{n + 1}_+ = \R^{n}\times [0,+\infty) \subset \R^{n+1}\), where \(F \cong \R^m\) is a Euclidean vector space (see for example \cite{JeffreyLee_2009}*{Def.\ 6.42}).
In this case, a covariant derivative can be expressed as \(\Diff_\Gamma = \Diff + \Gamma\) where \(\Gamma \colon \smash{\overline{\R}}^{n + 1}_+ \to \Lin{\R^{n+1}}{\mathfrak{o}(F)}\) is called a \emph{connection form}. In other words one sets for every \(x \in \smash{\overline{\R}}^{n + 1}_+\) and \(h \in \R^{n + 1}\),
\[
 \Diff_\Gamma u \brk{x}\sqb{h}
 = \Diff u \brk{x}\sqb{h} + \Gamma \brk{x}[h][u \brk{x}].
\]
The space \(\mathfrak{o}(F)\) is the Lie algebra associated to the group of isometries \(O \brk{F}\); the matrix representing \(\Gamma \brk{x}\) in an orthonormal basis of \(F\) is skew-symmetric.
This condition ensures the inequality
\(
  \norm{\Diff \norm{u}} \le \norm{\Diff_\Gamma u}\)  (\cref{lemma: diamagnetic inequality}).

In the case where \(\dim F = 2\), identify \(F \cong \C\), the connection form can be written as  \(\Gamma = \ii A\) for \(A\colon \smash{\overline{\R}}^{n + 1}_+ \to \R^{n+1}\), which corresponds to the magnetic vector potential.
Gauge-covariant Sobolev spaces arise naturally in settings of quantum
physics involving magnetic fields where they are called \emph{magnetic Sobolev spaces}~\citelist{\cite{Hall_2013}*{Ch.\ 23} \cite{Sontz_bundle}\cite{Balinsky_Evans_Lewis_2015}} and are a central tool in the studies of non-linear Schr\"odinger equations which have been the subject of extensive research~\citelist{\cite{Bonheure_Nys_VanSchaftingen_2019} \cite{Esteban_Lions_1989} \cite{Nguyen_Pinamonti_Squassina_Vecchi_2018}}.

Our first result generalises the recent work of Nguyen and the first author in which they obtained a characterisation of the corresponding magnetic Sobolev spaces \cite{Nguyen_VanSchaftingen_2020}.

\begin{theorem}\label{thm1}
Let \(n \geq 1\) and \(1 < p < +\infty\). If \(\Gamma \in C^1(\overline{\R}_+^{n+1}, \Lin{\R^{n+1}}{\mathfrak{o}(F)})\) and \(\spNorm{\mathscr K_\Gamma}{\infty} \leq \beta\), then there exists a positive constant \(C = C(n,p)>0\) depending only on \(n\) and \(p\) such that 
\begin{enumerate}[label=(\roman*)]
    \item For all \(U\in C^1_c(\smash{\smash{\overline{\R}}^{n + 1}_+},F)\)
    \begin{multline}\label{traces thm1 eq:1}
        \norm{U (\cdot, 0)}_{W^{1-1/p,p}_{\Gamma_\shortparallel}(\R^{n},F)} + \beta^{(p-1)/2p} \spNorm{U(\cdot,0)}{L^p(\R^n,F)}  \\
        \leq C (\spNorm{\nabla_\Gamma U}{L^p(\smash{\R^{n + 1}_+},\Lin{\R^{n+1}}{F})} + \beta^{1/2} \spNorm{U}{L^p(\smash{\R^{n + 1}_+},F )} ),
    \end{multline}
    \item For all \(u\in  C^1_c(\R^n,F)\), there exists \(U\in  C^1_c(\smash{\smash{\overline{\R}}^{n + 1}_+},F)\) depending linearly on \(u\) such that \(U(x,0)=u\brk{x}\) on \(\R^n\) and
    \begin{multline}\label{traces thm1 eq:2}
        \spNorm{\nabla_\Gamma U}{L^p(\smash{\R^{n + 1}_+},\Lin{\R^{n+1}}{F})} + \beta^{1/2} \spNorm{U}{L^p(\smash{\R^{n + 1}_+},F)} \\
        \leq C (\norm{u}_{W^{1-1/p}_{\Gamma_\shortparallel}(\R^n,F)} + \beta^{(p-1)/2p} \spNorm{u}{L^p(\R^n,F)}).
    \end{multline}
\end{enumerate}
\end{theorem}

The quantity \(\mathscr K_\Gamma\) in \cref{thm1} is the \emph{curvature} \(2\)-form of the connection induced by \(\Gamma\) and is defined as
\[
  \mathscr K_\Gamma = \dd \Gamma + \Gamma \wedge \Gamma
\]
 where \(\dd \Gamma\) is the exterior derivative of \(\Gamma\) and \(\wedge\) denotes the wedge product of differential forms; as explained in \cref{sect: preliminaries}, for every \(x \in \smash{\overline{\R}}^{n + 1}_+ \) \(v, w \in \R^{n + 1}\),
\[
  \mathscr K_\Gamma \brk{x}\sqb{v, w} = \dd \Gamma \brk{x}\sqb{v, w}+ \Gamma\brk{x}\sqb{v} \wedge \Gamma\brk{x}\sqb{w}.
\]
In the particular case where \(\dim F = 2\) so that \(F \cong \C\) and \(\mathfrak{o} \brk{F} \cong \ii \R\), writing \(\Gamma = \ii A\), we have 
\(\Gamma \brk{x}\sqb{v} \wedge \Gamma \brk{x}\sqb{w} = (\ii \wedge \ii) \brk{A \brk{x} \cdot v}\brk{A \brk{x} \cdot w} = 0\) and thus
\(
  \mathscr{K}_{\Gamma} = \dd \Gamma
\) and we recover the condition of \cite{Nguyen_VanSchaftingen_2020}.

The characterisation of \cref{thm1} relies on the \emph{fractional gauge-covariant Sobolev space} \(W^{s,p}_{\Gamma_\shortparallel}(\R^n,F)\) 
defined for \(1 \leq p < +\infty\), \(0<s<1\) and \(\Gamma_\shortparallel \in C^0(\R^n, \Lin{\R^n}{\mathfrak{o}(F)})\), as the set of all functions \(u\in L^p(\R^n,F)\) whose \emph{gauge-covariant Gagliardo seminorm}
\begin{equation}\label{gauge cov Gagliardo_1957}
	\norm u_{W^{s,p}_{\Gamma_\shortparallel}(\R^n, F)} = \biggl(\iint_{\R^n\times\R^n}\frac{\norm{u\brk{x} - R^{\Gamma_\shortparallel}\brk{x, y} u\brk{y}}^p}{\norm{x-y}^{n+sp}}\dd x \dd y\biggr)^{1/p},
\end{equation}
is finite,
where \(R^{\Gamma_\shortparallel}(x,y)\) denotes the parallel transport from \(y\) to \(x\) along the affine path \( t \in [0,1] \mapsto (1-t)x + t y\) induced by \(\Gamma_\shortparallel\) with \( \Gamma_\shortparallel\brk{x}[v] = \Gamma(x,0)[v,0]\).

An important feature of the statement of \cref{thm1} is its \emph{gauge-invariance}:
if \(\Gamma' = \phi \Diff_{\Gamma} \brk{\phi^{-1}}\) for some local isometry \(\phi \in C^1 \brk{\smash{\overline{\R}}^{n + 1}, O\brk{F}}\) which, in the context of vector bundles, is called a \emph{gauge transformation}, then 
\begin{align*}
 \Diff_{\Gamma'} \brk{\phi U} &= \phi \Diff_{\Gamma} U,\\
   \mathscr{K}_{\Gamma'} \phi
  &= \phi \mathscr{K}_{\Gamma}\\
\intertext{and}
  R^{\Gamma'_\shortparallel} \brk{x, y} \phi \brk{y} u\brk{y}
  &= \phi\brk{x} R^{\Gamma_\shortparallel}\brk{x, y} u\brk{y},
\end{align*}
(see \eqref{eq: D_Gamma gauge cov property}, \eqref{eq_curvature_gauge}, \cref{lemma: Pt gauge-invariance} and \eqref{eq_ooM3isahgu8ahghoh8ienizu}), so that \cref{thm1} is invariant under the change \(U \mapsto \phi U\), \(u \mapsto \phi u\) and \(\Gamma \mapsto \phi \Diff_{\Gamma} \brk{\phi^{-1}}\).

The proof of \cref{thm1} follows essentially the strategy of estimating along paths and extending through averaging that goes back to Gagliardo \cite{Gagliardo_1957}.
In order to deal with the covariant derivative in the abelian case, it was necessary to introduce phase shifts given by a path integral of the vector potential; even though those do not commute along paths, the deviation from commutativity could be rewritten thanks to the Stokes formula applied to the circulation of the vector potential in an integral of the curvature that could be estimated \cite{Nguyen_VanSchaftingen_2020}.
In order to prove \cref{thm1}, we replace the phase shifts obtained by integration by parallel transport; the deviance from commutativity is estimated thanks to holonomy estimates from \cite{Chanillo_VanSchaftingen_2020}.
The estimate on the trace and the construction of the extension are performed in \cref{prop traces,prop extension} respectively. 

By a density argument, see \cref{section: density}, we obtain the following
 characterisation of traces for gauge-covariant Sobolev mappings on the half-space, generalising \cite{Nguyen_VanSchaftingen_2020}*{Thm.\ 1.2} to the non-abelian case.

\begin{theorem}\label{trace thm on Rn}
   Let \(n\geq 1\) and \(1<p<+\infty\). If \(\Gamma\in  C^1(\smash{\overline{\R}}^{n + 1}_+, \Lin{\R^{n+1}}{\mathfrak{o}(F)} )\) and \(\spNorm{\mathscr{K}_\Gamma}{\infty} < +\infty\), then there exists a unique continuous linear and surjective operator
   \begin{equation*}
       \Tr\colon W^{1,p}_\Gamma(\smash{\overline{\R}}^{n+1}_+, F) \to W^{1-1/p,p}_{\Gamma_\shortparallel} (\R^n , F)
   \end{equation*}
   called the \emph{trace operator}, which satisfies \(\Tr U(\cdot) = U(\cdot,0)\) for all \(U\in C_c^1(\smash{\smash{\overline{\R}}^{n + 1}_+},F)\). Furthermore, there exists a linear and continuous \emph{extension operator}
   \begin{equation*}
       \Ext\colon W^{1-1/p,p}_{\Gamma_\shortparallel}(\R^n, F) \to W^{1,p}_\Gamma(\smash{\overline{\R}}^{n + 1}_+,F)
   \end{equation*}
   such that \(\Tr\circ \Ext\) is the identity map on \(W^{1-1/p,p}_{\Gamma_\shortparallel}(\R^n,F)\) and the estimates \eqref{traces thm1 eq:1} and \eqref{traces thm1 eq:2} of \cref{thm1} with \(u=\Tr U\) and \(U=\Ext u\) hold.
\end{theorem}
Finally, we obtain through a local trivialisation argument an analogue result when \(\mathcal{M}\) is a smooth compact Riemannian manifold with boundary \(\partial \mathcal{M}\) and \(\mathcal E\) is a smooth vector bundle above \(\mathcal{M}\) endowed with a metric connection \(K\). Precisely, we show that the fractional gauge-covariant Sobolev space space \( W^{1-1/p,p}_K(\partial \mathcal{M},\mathcal{E})\) consisting of all the sections \(u \in L^p(\partial\mathcal{M},\mathcal{E})\) which satisfy
\begin{equation}
	\smashoperator[r]{\iint_{\substack{\brk{x, y}\in\partial \mathcal{M} \times \partial \mathcal{M}\\ \operatorname{dist}_{\partial \mathcal{M}} \brk{x, y} < \mathrm{inj}_{\partial \mathcal{M}}}}} \frac{\norm{u\brk{x}-R^K_{\partial \mathcal{M}}\brk{x, y}u\brk{y}}^p}{\operatorname{dist}_{\partial \mathcal{M}}\brk{x, y}^{n+p-1}} \dd x  \dd y < +\infty
\end{equation} 
is the trace space of the gauge-covariant Sobolev space \(W_K^{1,p}(\mathcal{M},E)\). Here, \(\operatorname{dist}_{\partial \mathcal{M}}\) is the distance on \(\partial\mathcal{M}\), \(\mathrm{inj}_{\partial \mathcal{M}} > 0\) is the injectivity radius of \(\partial \mathcal{M}\) and \(R_{\partial \mathcal{M}}^K\brk{x, y}\) is the parallel transport (induced by the connection \(K\)) from \(y\) to \(x\) along a geodesic connecting the points. For more details, see \cref{sect: geometric tools}.

\section{Connection forms, covariant derivatives and transport}\label{sect: preliminaries}

\subsection{Covariant derivatives}
Given an open set \(\Omega \subseteq \R^{n+1}\) and a finite dimensional Euclidean vector space \(F \cong \R^m\) with inner product \((\cdot \vert \cdot)\) and induced norm \(\norm{}\),
a connection form \(\Gamma \colon \Omega \to \Lin{\R^{n+1}}{\mathfrak{o}(F)}\) induces a covariant derivative \(\Diff_\Gamma = \Diff + \Gamma\) on \(\Omega\) defined for any smooth function \(U\colon \Omega \to F\), \(x \in \Omega\) and \(v \in \R^{n+1}\) by 
\begin{equation}
\label{eq_def_DGamma}
 	\Diff_\Gamma U\brk{x} [v] = \Diff U\brk{x}[v] + \Gamma\brk{x}[v]U\brk{x}.
\end{equation}
If \(U^\prime = \phi U\) for a gauge transformation \(\phi\in C^1(\Omega, O(F))\), one computes
\begin{equation}
	\label{eq: computation for Gamma^prime}
	\begin{split}
		\Diff_\Gamma U\brk{x} 
		& = \Diff(\phi\brk{x}^{-1}U^\prime\brk{x}) + \Gamma\brk{x}\phi\brk{x}^{-1}U^\prime\brk{x} \\
		& =  \phi\brk{x}^{-1} \Diff U^\prime\brk{x} + (\Gamma\brk{x}\phi\brk{x}^{-1}+ \Diff(\phi\brk{x}^{-1}))U^\prime\brk{x} \\
		& = \phi\brk{x}^{-1}\Diff U^\prime\brk{x} + \phi\brk{x}^{-1}(\phi\brk{x}\Gamma\brk{x}\phi\brk{x}^{-1} + \phi\brk{x}\Diff(\phi\brk{x}^{-1})) U^\prime\brk{x}.
	\end{split}
\end{equation}
If we define a new connection form \(\Gamma^\prime \colon \Omega \to \Lin{\R^{n+1}}{\mathfrak{o}(F)}\) by
\begin{equation}
	\label{eq: Gamma^prime}
	\Gamma^\prime = \phi \Diff_{\Gamma} (\phi^{-1})
	= \phi\Diff(\phi^{-1}) + \phi \Gamma \phi^{-1} 
	= -(\Diff \phi) \phi^{-1} + \phi \Gamma \phi^{-1}
\end{equation}
where \(\phi^{-1} \brk{x} = \phi \brk{x}^{-1} \in O (F)\) is the pointwise inverse of \(\phi \brk{x} \in O(F)\), the corresponding covariant derivative \(\Diff_{\Gamma^\prime}\) given by \eqref{eq_def_DGamma} has the \emph{gauge-covariance property}: if \(U^\prime = \phi U\), then 
\begin{equation}
	\label{eq: D_Gamma gauge cov property}
  \Diff_{\Gamma^\prime} U^\prime = \phi \Diff_{\Gamma} U.
\end{equation}

The \emph{weak covariant derivative}, also denoted by \(\Diff_\Gamma\), is defined analogously to \eqref{eq_def_DGamma}: for any weakly differentiable \(U \in W^{1,1}_{\mathrm{loc}}(\Omega,F)\) with weak derivative \(\Diff U\), the weak covariant derivative \(\Diff_\Gamma U\colon \Omega \to \Lin{\R^{n+1}}{F}\) of \(U\) is given by
\begin{equation}
	\label{eq_def_weak_DGamma}
	\Diff_\Gamma U = \Diff U + \Gamma U.
\end{equation}
The weak covariant derivative also satisfies the gauge-covariant property \eqref{eq: D_Gamma gauge cov property}.

When \(\Gamma \in L^1_{\mathrm{loc}}(\Omega, \Lin{\R^{n+1}}{\mathfrak{o}(F)})\), we define for \(1 \leq p < +\infty\) the \emph{first order gauge-covariant Sobolev space}
\begin{equation}
	\label{def: W1p_Gamma}
	W^{1,p}_\Gamma(\Omega, F) = \set[\Big]{U \in W^{1,1}_\mathrm{loc}(\Omega, F) \st \int_\Omega \norm{U}^p + \norm{\Diff_\Gamma U}^p < + \infty}.
\end{equation}
The gauge-covariant Sobolev space \(W^{1,p}_\Gamma(\Omega,F)\) forms a Banach space (and a Hilbert space when \(p=2\)) when endowed with the norm
\begin{equation*}
	\Norm{U}_{W^{1,p}_\Gamma(\Omega,F)} 
	= \brk[\bigg]{\int_\Omega \norm{U}^p + \norm{\Diff_\Gamma U}^p}^{1/p} 
	= \Bigl(\Norm{U}_{L^p(\Omega,F)}^p + \Norm{\Diff_\Gamma U}_{L^p(\Omega,\Lin{\R^{n+1}}{F})}^p\Bigr)^{1/p}.
\end{equation*}
The gauge-covariant Sobolev spaces \(W^{1,p}_\Gamma\) are gauge-invariant in the sense that if \(\Gamma^\prime\) is given by \eqref{eq: Gamma^prime}, then \(W^{1,p}_{\Gamma^\prime}=W^{1,p}_\Gamma\).
To fix the ideas, the norm used on \(\Diff_\Gamma U\brk{x}\) is here and in what follows the Hilbert--Schmidt (or Frobenius) norm on \(\Lin{\R^{n+1}}{F}\)~\citelist{\cite{Adams_Fournier_2003}*{Def. 6.57}},
\begin{equation*}
 \norm{\Diff_{\Gamma} U \brk{x}} = \sqrt{\operatorname{tr}(\Diff_{\Gamma} U \brk{x}^\ast\Diff_{\Gamma} U \brk{x})}.
\end{equation*}

Since 
\begin{equation}
\label{eq: metric compatible}
\inpr{\Diff_\Gamma U_1 }{ U_2} + \inpr{U_1 }{ \Diff_\Gamma U_2} = \Diff\inpr{ U_1 }{ U_2} + \inpr{\Gamma U_1 }{ U_2} + \inpr{U_1 }{ \Gamma U_2},
\end{equation}
the condition that \(\Gamma\) is valued in the Lie algebra \(\mathfrak{o}(F)\) of the orthogonal group \(O(F)\) is equivalent to the requirement that the associated covariant derivative \(\Diff_\Gamma\) is a \emph{metric connection} (or \emph{compatible with the metric}), that is, for all \(U_1, U_2 \in W^{1,1}_{\mathrm{loc}}(\Omega,F)\) (see for example~\cite{DifferentialGeometryTu}*{\S 10.5})
\begin{equation*}
	\Diff \inpr{U_1}{U_2} =
	\inpr{\Diff_\Gamma U_1 }{ U_2}
	+ \inpr{U_1 }{ \Diff_\Gamma U_2}.
\end{equation*}
As a consequence we get the following inequality.

\begin{lemma}\label{lemma: diamagnetic inequality}
	Let \(\Gamma \colon \Omega \to \Lin{\R^{n+1}}{\mathfrak{o}(F)}\) and let \(U \in W^{1,1}_{\mathrm{loc}}(\Omega, F)\). For almost all \(x \in \Omega\) and \(v \in \R^{n + 1}\), we have
	\begin{equation*}
		\norm{\Diff \norm{U}\brk{x}\sqb{v}} \leq \norm{\Diff_\Gamma U\brk{x}\sqb{v}}.
	\end{equation*}
\end{lemma}

In the magnetic setting \(\Gamma = \ii A\), the above inequality is known as the \emph{diamagnetic inequality}~\citelist{\cite{Lieb_Loss_2001}*{Thm.\ 7.21}\cite{Balinsky_Evans_Lewis_2015}*{\S 5.3}}. When \(\Gamma = 0\) is trivial, \cref{lemma: diamagnetic inequality} reduces to the well-known inequality \(\norm{\Diff \norm{U}} \leq \norm{\Diff U}\) for vector-valued Sobolev functions (see for example~\cite{Willem_2013}*{Cor.\ 6.1.14}).

\begin{proof}[Proof of \cref{lemma: diamagnetic inequality}]
	Recall that \(\norm{U}\) is weakly differentiable whenever \(U\) is and that for almost all \(x\in \Omega\)
	\begin{equation}
		\label{eq: weak derivative |U|}
		\Diff\norm{U}\brk{x}\sqb{v}= \left\{
		\begin{aligned}
			& \frac{\inpr{U\brk{x} }{ \Diff U\brk{x}\sqb{v}}}{\norm{U\brk{x}}} &&\text{if }U\brk{x} \neq 0,\\
			& 0 &&\text{if }U\brk{x}=0.
		\end{aligned}\right.
	\end{equation}
	We have, since \(\Gamma \brk{x}\sqb{v} \in \mathfrak{o}\brk{F}\),
	\begin{equation}
	\begin{split}
	  \inpr{U\brk{x} }{ \Diff U\brk{x}\sqb{v}}
	  &=\inpr{U \brk{x} }{ \Diff U \brk{x}[v]} + \inpr{U \brk{x} }{ \Gamma\brk{x}[v] U \brk{x}}\\
	  &=\inpr{U\brk{x} }{ \Diff_\Gamma U\brk{x}\sqb{v}},
	\end{split}
	\end{equation}
	and thus by the Cauchy-Schwarz inequality
	\begin{equation}
	  \norm{\inpr{U\brk{x} }{ \Diff U\brk{x}\sqb{v}}}
	  \le \norm{U\brk{x}} \norm{\Diff_\Gamma U\brk{x}\sqb{v}},
	\end{equation}
	and the conclusion follows. 
\end{proof}

A direct consequence of \cref{lemma: diamagnetic inequality} is that if \(U \in W^{1,p}_\Gamma(\Omega,F)\), then \(\norm{U} \in W^{1,p}(\Omega,\R)\) and \(\smash{\Norm{U}_{W^{1,p}_\Gamma(\Omega,F)}} \leq \Norm{ \norm{U} }_{W^{1,p}(\Omega,\R)}\).
In particular, it also implies the Sobolev-type embedding \(W^{1,p}_\Gamma(\Omega,F) \hookrightarrow \smash{ L^{p^\ast}(\Omega,F)}\) if \(n+1>p\) where \(\smash{\frac{1}{p^\ast}} = \frac{1}{p} - \frac{1}{n + 1}\).

\subsection{Parallel transport along paths}
While the covariant derivative has the gauge-covariance property \eqref{eq: D_Gamma gauge cov property}, differences between values at points do not transform well.
Indeed, if \(x\), \(y \in \Omega\) and \(U^\prime = \phi U\) for a gauge transformation \(\phi\), then
\begin{equation*}
	U^\prime\brk{x} - U^\prime\brk{y} = \phi\brk{x} U\brk{x} - \phi\brk{y} U\brk{y}
\end{equation*} 
which cannot be rewritten as \(\phi(\cdot) (U\brk{x}- U\brk{y})\) unless \(\phi\brk{x}=\phi\brk{y}\).
One thus has to work with a \emph{parallel transport} (or \emph{parallel displacement}) with respect to the connection~\citelist{\cite{Kobayashi_Nomizu_1963} \cite{Sontz_bundle} \cite{Wendl_bundles}}. 

When \(\Gamma \in C^0(\Omega, \Lin{\R^{n+1}}{\mathfrak{o}(F)})\), any smooth path \(\gamma\colon[0, 1] \to \Omega\) naturally induces a parallel transport operator \( \Pt^\Gamma_{\gamma} \colon[0, 1] \to O(F)\) which is defined as the solution of the linear differential equation
\begin{equation}\label{def: Pt}
	\Biggl\{\begin{aligned}
		&\bigl(\Pt^\Gamma_\gamma\brk{t}\bigr)^\prime + \Gamma(\gamma\brk{t})[\dot{\gamma}\brk{t}] \Pt^\Gamma_\gamma \brk{t} = 0  &&\text{for \(t \in[0, 1]\),}\\
		&\Pt^\Gamma_\gamma(1) = \id_F.
	\end{aligned}
\end{equation}
The classical theory of linear differential equations ensures that \(\Pt^\Gamma_\gamma\) is well-defined and as smooth as data allows it to be. 

In the particular case where \(F \cong \C\) and \(\Gamma = \ii A\), the parallel transport can be simply computed through a path integral of the vector potential \(A\)
\[
 \Pt^\Gamma_\gamma \brk{t} = \exp \brk[\bigg]{-\ii \int_t^1 A \brk{\gamma} \sqb{\dot{\gamma} \brk{\tau}} \dd \tau}.
\]

\begin{lemma}
	\label{lemma: Pt isometry}
	If \(\Gamma \in C^0(\Omega,\Lin{\R^{n+1}}{\mathfrak{o}(F)})\), \(\gamma \in C^1([0, 1], \Omega)\) and \(t \in[0, 1]\), then
	\(\Pt^\Gamma_{\gamma}\brk{t} \in O \brk{F}\).
\end{lemma}
\begin{proof}
Given \(v \in F\) and defining \(f (t) = \norm{\Pt^\Gamma_{\gamma} \brk{t}}^2\), 
we have 
\begin{equation}
\begin{split}
f' \brk{t} &= 2 \inpr{\Pt^\Gamma_{\gamma} \brk{t} }{ \brk{\Pt^\Gamma_{\gamma}}'\brk{t}}\\
&= 2\inpr{\Pt^\Gamma_{\gamma} \brk{t}}{ -\Gamma(\gamma\brk{t}}[\dot{\gamma} \brk{t}] \Pt^\Gamma_\gamma \brk{t}) = 0,
\end{split}
\end{equation}
and the conclusion follows.
\end{proof}

We have the following gauge-covariance property for parallel transport.

\begin{lemma}
	\label{lemma: Pt gauge-invariance}
	Let \(\gamma \colon[0, 1] \to \Omega\) be absolutely continuous and let the connection form \(\Gamma \in C^0(\Omega,\Lin{\R^{n+1}}{\mathfrak{o}(F)})\). If \(\phi \in C^1\brk{\Omega, O(F)}\) is a gauge transformation and if \(\Gamma^\prime\) is given by \eqref{eq: Gamma^prime}, then for all \(t\in[0, 1]\) one has
	\begin{equation*}
		\Pt^{\Gamma^\prime}_\gamma(t) = \phi(\gamma(t))\Pt^\Gamma_\gamma(t)\phi(\gamma(1))^{-1}
	\end{equation*}
	where \(\Pt^\Gamma_\gamma\) and \(\Pt^{\Gamma^\prime}_\gamma\) are defined by \eqref{def: Pt} with their respective connection forms.
\end{lemma}
\begin{proof}
	One verifies that \( t\in[0, 1] \mapsto \phi(\gamma(t))\Pt^\Gamma_\gamma(t)\phi(\gamma(1))^{-1}\) satisfies the same final value problem \eqref{def: Pt} as \(\Pt^{\Gamma^\prime}_\gamma\), and the conclusion follows by uniqueness of the solution.
\end{proof}

\Cref{lemma: Pt gauge-invariance} implies that if \(U^\prime = \phi U\), then for any absolutely continuous path \(\gamma \colon[0, 1] \to \Omega\) joining \(x=\gamma(0)\) and \(y=\gamma(1)\), we have
\begin{equation}
	\label{eq: Pt U - U}
	\begin{split}
	U^\prime\brk{x} - \Pt^{\Gamma^\prime}_\gamma(0)U^\prime\brk{y}
	& = \phi\brk{x}U\brk{x} - \phi\brk{x}\Pt^\Gamma_\gamma(0)\phi\brk{y}^{-1}\phi\brk{y} U\brk{y} \\
	& = \phi\brk{x} (U\brk{x} - \Pt^\Gamma_\gamma(0)U\brk{y}).
	\end{split}
\end{equation}

With this notation at hand,  we state a gauge-covariant equivalent of the fundamental theorem of calculus, namely
\begin{lemma}
	\label{FTC}
	For every \(\gamma \in C^1([0, 1], \Omega)\) and every \(\Gamma \in C^0(\Omega,\Lin{\R^{n+1}}{\mathfrak{o}(F)})\), we have 
	\begin{equation}
		U(\gamma(1)) = \Pt^\Gamma_\gamma\brk{0}^{-1} U(\gamma(0)) + \int_0^1 \Pt^\Gamma_\gamma \brk{t}^{-1}\Diff_\Gamma
		U(\gamma(t))[\dot{\gamma} \brk{t}]\dd t.
	\end{equation}
\end{lemma}
\Cref{FTC} is the counterpart of the magnetic fundamental theorem of calculus~\cite{Nguyen_VanSchaftingen_2020}*{(2.6)}.

\begin{proof}[Proof of \cref{FTC}]
	By the chain rule and the definition of the parallel transport \eqref{def: Pt}, we have, for all \(t\in[0,1]\)
	\begin{equation}\label{with chaine rule}
	\begin{split}
		\frac{\dd}{\dd t}\brk[\big]{\Pt^\Gamma_\gamma\brk{t}^{-1} U(\gamma(t))}
		&= \Pt^\Gamma_\gamma\brk{t}^{-1} \Gamma(\gamma(t))[\dot{\gamma}\brk{t}]U(\gamma(t)) + \Pt^\Gamma_\gamma\brk{t}^{-1} \Diff U(\gamma)[\dot{\gamma}\brk{t}] \\
		&= \Pt^\Gamma_\gamma\brk{t}^{-1} \Diff_\Gamma U(\gamma(t))[\dot{\gamma}\brk{t}].
	\end{split}
	\end{equation}
	and the conclusion follows from the fundamental theorem of calculus.
\end{proof}

Based on the technique presented by Chanillo and the first author in~\cite{Chanillo_VanSchaftingen_2020}, we obtain a geometric formula for the derivative with respect to a parameter, namely   
\begin{lemma}
	\label{lemma: formula for derivative of Pt wrt parameter}
	Let \(\Gamma\in C^1(\Omega,\Lin{\R^{n+1}}{\mathfrak o(F)})\). Let \(H \in C^1([0,1]\times J, \Omega)\) where \(J\subset \R\) is open and set \(\gamma_s = H(\cdot , s) \in C^1([0,1], \Omega)\) for \(s \in J\). Then \((t,s) \in [0,1] \times J \mapsto \Pt^\Gamma_{\gamma_s}(t) \in C^1([0,1]\times J, O(F))\) and, for all \(t\in [0,1]\) and \(s \in J\),
	\begin{equation}
		\label{eq: lemme: formula for derivative of Pt wrt parameter}
		\begin{split}
			&\tfrac{\dd}{\dd s}\Pt^\Gamma_{\gamma_s}(0)
			+ \Gamma(\gamma_s(0))[\tfrac{\dd}{\dd s} \gamma_s(0)]\Pt^\Gamma_{\gamma_s}(0)
			- \Pt^\Gamma_{\gamma_s}(0) \Gamma(\gamma_s(1))[\tfrac{\dd}{\dd s}\gamma_s(1)] \\
			&\qquad = \int_0^1 \Pt^\Gamma_{\gamma_s}(0) \Pt^\Gamma_{\gamma_s}(t)^{-1} \mathscr{K}_\Gamma(\gamma_s(t))[\tfrac{\dd}{\dd s} \gamma_s(t), \dot{\gamma}_s(t)] \Pt^\Gamma_{\gamma_s}(t) \dd t.
		\end{split}
	\end{equation}
and 
\begin{equation}
\label{eq_CaaPhah3kae2shei7oYeiLah}
\begin{split}
&\norm[\Big]{\tfrac{\dd}{\dd s}\Pt^\Gamma_{\gamma_s}(0) 
	+ \Gamma(\gamma_s(0))[\tfrac{\dd}{\dd s}\gamma_s(0)] \Pt^\Gamma_{\gamma_s}(0)
		-\Pt^\Gamma_{\gamma_s}(0)\Gamma(\gamma_s(1))[\tfrac{\dd}{\dd s} \gamma_s(1)]}\\
&\qquad \le \int_0^1 \norm{\mathscr{K}_\Gamma(\gamma_s(t))[\tfrac{\dd}{\dd s} \gamma_s(t), \dot{\gamma}_s(t)]} \dd t.
		\end{split}
\end{equation}
\end{lemma}
The term \(\mathscr{K}_\Gamma\) in \cref{lemma: formula for derivative of Pt wrt parameter} is the \emph{curvature} of the covariant derivative \(\Diff_\Gamma\) and is defined for all \(x\in \Omega\), \(v\), \(w\in \R^{n+1}\) by~\citelist{\cite{Kobayashi_Nomizu_1963}\cite{Chanillo_VanSchaftingen_2020}\cite{Wendl_bundles}\cite{Sontz_bundle}}
\begin{equation}
	\label{def: curvature}
	\mathscr{K}_\Gamma\brk{x}\sqb{v, w} = \dd \Gamma\brk{x}\sqb{v, w} + \Gamma \wedge \Gamma\brk{x}\sqb{v, w}
\end{equation}
where \(\dd \Gamma\) is the \emph{exterior derivative} of \(\Gamma\) (seen as a \(\mathfrak{o}(F)\) valued form \(\Gamma \in \bigwedge^1(\Omega,\mathfrak{o}(F))\)) and \(\Gamma \wedge \Gamma\) is the \emph{wedge product} of \(\mathfrak{o}(F)\) valued forms. In other words,
\begin{equation}
\label{eq_booxiNie6ni5luu2arah3Shu}
	\mathscr{K}_\Gamma\brk{x}\sqb{v, w} 
	= \Diff\Gamma\brk{x}\sqb{v, w} - \Diff\Gamma\brk{x}[w,v]  + \Gamma\brk{x}[v]\Gamma\brk{x}[w] - \Gamma\brk{x}[w]\Gamma\brk{x}[v],
\end{equation}
where 
\[
	\Diff\Gamma\brk{x}\sqb{v, w} = \left.\frac{\dd}{\dd t}\right\vert_{t=0} \Gamma(x+tv)[w]
\]
One can also write 
\begin{equation}
\label{eq_kifaelaer1feiThohV1aesh2}
 \mathscr{K}_\Gamma\brk{x}\sqb{v, w}
 = \Diff_\Gamma \brk{\Gamma\sqb{w}}\brk{x}\sqb{v} - \Diff_\Gamma \brk{\Gamma\sqb{v}}\brk{x}\sqb{w},
\end{equation}
which shows that the curvature is related to the commutation of derivatives.
Indeed, one has 
\[
 \Diff_{\Gamma} \brk{\Diff_{\Gamma} U\sqb{w}} \sqb{v}
  = \Diff^2 U \sqb{v, w} + \Gamma\sqb{v} \Diff U \sqb{w}
  + \Gamma\sqb{v}\Gamma\sqb{w} U + \Diff \Gamma \sqb{v,w} U
  + \Gamma\sqb{w} \Diff U \sqb{v},
\]
and thus 
\begin{equation}
\label{eq_ohvi0ohloo2weiBu8yoyahhe}
 \Diff_{\Gamma} \brk{\Diff_{\Gamma} U\sqb{w}} \sqb{v}
 -  \Diff_{\Gamma} \brk{\Diff_{\Gamma} U\sqb{v}} \sqb{w}
 = \mathscr{K}_\Gamma \sqb{v, w} U.
\end{equation}

If \(\Gamma'\) is given by \eqref{eq: Gamma^prime},
then, by \eqref{eq: D_Gamma gauge cov property} and \eqref{eq_ohvi0ohloo2weiBu8yoyahhe}
\begin{equation}
\label{eq_curvature_gauge}
\begin{split}
\mathscr{K}_{\Gamma'}\sqb{v, w}
&=  \Diff_{\Gamma'} \brk{\phi \Diff_{\Gamma} \brk{\phi^{-1}} \sqb{w}}\sqb{v} - \Diff_{\Gamma'} \brk{\phi \Diff_{\Gamma} \brk{\phi^{-1}}\sqb{v}}\sqb{w}\\
&= \phi \brk{\Diff_{\Gamma} \brk{\Diff_{\Gamma} \brk{\phi^{-1}} \sqb{w}}\sqb{v} - \Diff_{\Gamma} \brk{\Diff_{\Gamma} \brk{\phi^{-1}}\sqb{v}}\sqb{w}}\\
&= \phi \mathscr{K}_{\Gamma}\sqb{v, w} \phi^{-1},
\end{split}
\end{equation}
that is, \(\mathscr{K}_{\Gamma'} = \phi \mathscr{K}_{\Gamma}\phi^{-1}\).

In the case of abelian gauge theories, such as for magnetic Sobolev spaces, the wedge product \(\Gamma \wedge \Gamma = 0\) vanishes, so that the curvature \eqref{def: curvature} is given by \(\mathscr{K}_\Gamma = \dd \Gamma\). For magnetic Sobolev spaces, where \(\Gamma = \ii A\), one has \(\mathscr{K}_{\ii A} = \ii \dd A\) so that the curvature is the underlying magnetic field \(B = \dd A\) up to a constant factor \(\ii\)~\citelist{\cite{Nguyen_VanSchaftingen_2020}\cite{Lieb_Loss_2001}\cite{Sontz_bundle}}.

The identity \eqref{eq: lemme: formula for derivative of Pt wrt parameter} becomes the formula for the differentiation of line integrals connected to the Reynolds transport formula and Kelvin’s circulation theorem (see for example \cite{Flanders_1973}),
\begin{multline}
 \tfrac{\dd }{\dd s}
 \int_0^1 A \brk{\gamma_s \brk{t}}\sqb{\dot{\gamma}_s \brk{t}} \dd t
  - A\brk{\gamma_s \brk{0}}\sqb{\tfrac{\dd }{\dd s} \gamma_s\brk{0}} + A\brk{\gamma_s \brk{1}}\sqb{\tfrac{\dd }{\dd s} \gamma_s \brk{1}}\\
  =
  \int_0^1 \dd A \brk{\gamma_s \brk{t}} \sqb{\dot{\gamma}_s \brk{t}, \tfrac{\dd }{\dd s} \gamma_s \brk{t}} \dd t.
\end{multline}

\begin{proof}[Proof of \cref{lemma: formula for derivative of Pt wrt parameter}]
	Since \(\Gamma \in C^1(\Omega, \Lin{\R^{n+1}}{\mathfrak{o}(F)})\), the theory of ordinary differential equations with parameters ensures that \((t,s) \mapsto \Pt^\Gamma_{\gamma_s}(t)\) is \(C^1\) and satisfies~\cite{Kong_2014}*{Thm.\ 1.5.3}
	\begin{equation}
		\label{eq: d/dalpha Pt_gamma_alpha}
		\left\{\begin{aligned}
			 &(\tfrac{\dd}{\dd s} \Pt^\Gamma_{\gamma_s})^\prime(t) + \Gamma(\gamma_s(t))[\dot{\gamma}_s(t)]\tfrac{\dd}{\dd s}\Pt^\Gamma_{\gamma_s}(t) \\
			& \hspace{3em} = - ( \Diff \Gamma(\gamma_s(t))[ \tfrac{\dd}{\dd s} \gamma_s(t), \dot{\gamma}_s(t) ] + \Gamma(\gamma_s(t))[\tfrac{\dd}{\dd s}\dot{\gamma}_s(t)]) \Pt^\Gamma_{\gamma_s}(t), \\
			&\tfrac{\dd}{\dd s} \Pt^\Gamma_{\gamma_s}\brk{1} = 0.
		\end{aligned}\right.
	\end{equation}
	The final value problem \eqref{eq: d/dalpha Pt_gamma_alpha} is a non-homogeneous linear differential equation and, since \(\Pt^\Gamma_{\gamma_s}\) is an operator solution of the associated homogeneous linear differential equation, applying the variation of parameters formula for such equations~\cite{Kong_2014}*{Thm.\ 2.3.1} yields
	\begin{multline}
		\label{eq: sol via variation of parameters}
		\tfrac{\dd}{\dd s} \Pt^\Gamma_{\gamma_s}(0) \\
		\qquad =
		\int_0^1\Pt^\Gamma_{\gamma_s}(0) \Pt^\Gamma_{\gamma_s}(t)^{-1}
		\brk[\big]{ \Diff \Gamma(\gamma_s(t))[\tfrac{\dd}{\dd s} \gamma_s(t), \dot{\gamma}_s(t) ] + \Gamma(\gamma_s(t))[\tfrac{\dd}{\dd s} \dot{\gamma}_s(t)]} \Pt^\Gamma_{\gamma_s}(t) \dd t.		\end{multline}
	Since \((\Pt^\Gamma_{\gamma_s}(t)^{-1})^\prime = \Pt^\Gamma_{\gamma_s}(t)^{-1}\Gamma(\gamma_s(t))[\dot{\gamma}_s(t)]\), an integration by parts yields
	\begin{equation}
		\label{eq: integration by parts}
		\begin{split}
		&\int_0^1 \Pt^\Gamma_{\gamma_s}\brk{0} \Pt^\Gamma_{\gamma_s}(t)^{-1} \Gamma(\gamma_s(t))[\tfrac{\dd}{\dd s} \dot{\gamma}_s(t)] \Pt^\Gamma_{\gamma_s}(t)\dd t \\
		&\qquad 	= \Pt^\Gamma_{\gamma_s}\brk{0}  \Gamma(\gamma_s(1))[\tfrac{\dd}{\dd s}\gamma_s(1)] -  \Gamma(\gamma_s(0))[\tfrac{\dd}{\dd s} \gamma_s(0)] \Pt^\Gamma_{\gamma_s}(0)\\
		&\qquad \qquad + \int_0^1 \Pt^\Gamma_{\gamma_s}(0) \Pt^\Gamma_{\gamma_s}(t)^{-1}\Gamma(\gamma_s(t))[\dot{\gamma}_s(t)] \Gamma(\gamma_s(t))[\tfrac{\dd}{\dd s}\gamma_s(t)]\Pt^\Gamma_{\gamma_s}(t) \dd t \\
		&\qquad \qquad - \int_0^1 \Pt^\Gamma_{\gamma_s}(0) \Pt^\Gamma_{\gamma_s}(t)^{-1} \Diff\Gamma(\gamma_s(t))[\dot{\gamma}_s(t), \tfrac{\dd}{\dd s}\gamma_s(t)] \Pt^\Gamma_{\gamma_s}(t)\dd t \\
		&\qquad \qquad
			- \int_0^1 \Pt^\Gamma_{\gamma_s}(0)\Pt^\Gamma_{\gamma_s}(t)^{-1} \Gamma(\gamma_s(t))[\tfrac{\dd}{\dd s}\gamma_s(t)] \Gamma(\gamma_s)[\dot{\gamma}_s(t)]\Pt^\Gamma_{\gamma_s}(t) \dd t.
		\end{split}
	\end{equation}
	Injecting \eqref{eq: integration by parts} into \eqref{eq: sol via variation of parameters}, we reach the conclusion \eqref{eq: lemme: formula for derivative of Pt wrt parameter} by the definition of curvature \eqref{eq_booxiNie6ni5luu2arah3Shu}.
\end{proof}

In \cref{lemma: formula for derivative of Pt wrt parameter}, the proof of the identity \eqref{eq: lemme: formula for derivative of Pt wrt parameter} remains valid for a general connection \(\Gamma \in C^1 \brk{\Omega, \Lin {\R^{n+1}}{ \mathfrak{gl} \brk{F}}}\) and the isometric character is only used to get \eqref{eq_CaaPhah3kae2shei7oYeiLah}.

\subsection{Parallel transport along segments}

Assuming now that \(\Omega \subseteq \R^{n+1}\) is convex, given \(x\), \(y\in \Omega\), we can consider the path 
\begin{equation}
	\label{eq: def gammaxy}
	\gamma_{x,y} \colon [0,1] \to \Omega: t \mapsto (1-t)x + ty
\end{equation}
and define \(R^{\Gamma} \colon \Omega \times \Omega \to O \brk{F}\) by 
\begin{equation}
	\label{eq: def R}
	R^\Gamma\brk{x, y} = \Pt^\Gamma_{\gamma_{x,y}}(0) \in O(F),
\end{equation}
the parallel transport \emph{from} \(y\) \emph{to} \(x\) along the affine path \(\gamma_{x,y}\). 
In the particular case where \(F \cong \C\) and \(\Gamma = \ii A\), we have
\begin{align}
 R^{\Gamma} \brk{x, y} &= \e^{-\ii \mathcal{I}^A \brk{x, y}}&
 &\text{where }&
 \mathcal{I}^A \brk{x, y}
 &= \int_0^1 A \brk{\brk{1 -t} x + t y}\sqb{x - y} \dd t.
\end{align}

It follows from \cref{lemma: Pt gauge-invariance} that 
if \(\Gamma'\) is given by \eqref{eq: Gamma^prime}, then 
\begin{equation}
\label{eq_ooM3isahgu8ahghoh8ienizu}
 R^{\Gamma'} \brk{x, y}\, \phi \brk{y} =
 \phi \brk{x}\, R^{\Gamma} \brk{x, y}.
\end{equation}
By \cref{FTC}, we have 
\begin{equation}
	U\brk{y} - R^\Gamma\brk{x, y}^{-1}U\brk{x} = \int_0^1 \sqb{R^\Gamma\brk{x, \brk{1 - t} x + t y}}^{-1}\Diff_\Gamma
	U(\brk{1 -t} x + t y)[y-x]\dd t,
\end{equation}
and thus
\begin{equation}
  \norm{U\brk{x} - R^\Gamma\brk{x, y}U\brk{y}}
  \le
  \int_0^1 \norm{\Diff_\Gamma
	U(\brk{1 -t} x + t y)[y-x]}\dd t,
\end{equation}
Moreover,
\begin{equation}
	\label{eq: Rxy = Ryx-1}
	R^\Gamma\brk{x, x} = \id_F 
\end{equation}
and, if the points \(x, y, z\) are \emph{colinear}, 
\begin{equation}
	\label{eq: R=Pt Pt-1}
   R^\Gamma\brk{x, y}R^\Gamma\brk{y, z} R^\Gamma\brk{z, x}=\id_F.
\end{equation}

If \(I\subseteq \R\) is an interval and if \(s \in I \mapsto \brk{x_s, y_s} \in \Omega\) is differentiable, it follows from \cref{lemma: formula for derivative of Pt wrt parameter} applied to \(\gamma_s = \gamma_{x_s, y_s}\) that
\begin{equation}
\label{eq: d/ds R(xs,ys)}	
\begin{split}
  &\frac{\dd}{\dd s}
  R^\Gamma \brk{x_s, y_s} + \Gamma(x_s)\sqb[\big]{\tfrac{\dd x_s}{\dd s}} R^\Gamma \brk{x_s, y_s} -
  R^\Gamma \brk{x_s, y_s} \Gamma(y_s)\sqb[\big]{\tfrac{\dd y_s}{\dd s}}\\
  & \qquad = \int_0^1 R^\Gamma \brk{x_s, \brk{1-t}x_s + t y_s} \mathscr{K}_\Gamma(\brk{1-t}x_s + t y_s)[\brk{1-t}\tfrac{\dd x_s}{\dd s} + t \tfrac{\dd y_s}{\dd s},  y_s- x_s]\\
  &\qquad \qquad \qquad \qquad R^\Gamma \brk{\brk{1-t}x_s + t y_s, y_s}\dd \tau,
\end{split}
\end{equation}

and thus 
\begin{equation}
\label{eq_uyooGuezaik5icei3UCaeh3i}
\begin{split}
  &\norm[\Big]{\frac{\dd}{\dd s}
  R^\Gamma \brk{x_s, y_s} + \Gamma(x_s)\sqb[\big]{\tfrac{\dd x_s}{\dd s}} R^\Gamma \brk{x_s, y_s} -
  R^\Gamma \brk{x_s, y_s} \Gamma(y_s)\sqb[\big]{\tfrac{\dd y_s}{\dd s}}}\\
 &\qquad \le \int_{0}^{1} \norm{\mathscr{K}_\Gamma(\brk{1-t}x_s + t y_s)[\brk{1-t}\tfrac{\dd x_s}{\dd s} + t \tfrac{\dd y_s}{\dd s},  y_s - x_s]} \dd t.
\end{split}
\end{equation}

\subsection{Estimates of holonomy}
When \(x, y, z\) are not colinear, there is no reason for \eqref{eq: R=Pt Pt-1} to hold.
It corresponds to the  \emph{holonomy} at \(x\) along a the triangle \(\Delta_{x,y,z}\) with vertices \(x\), \(y\), \(z\)~\cite{Kobayashi_Nomizu_1963}.
The next proposition controls how \eqref{eq: R=Pt Pt-1} fails, providing a piecewise affine counterpart to the estimate in the smooth case \cite{Chanillo_VanSchaftingen_2020}.

\begin{proposition}\label{estimate holonomy triangle}
    Let \(\Gamma \in  C^1( \Omega, \Lin{\R^{n+1}}{\mathfrak{o}(F)})\). 
    If \(\Bar{\Delta}_{x, y, z} \subseteq \Omega\), then 
    \begin{equation*}
        \norm{\id_F - R^\Gamma \brk{x, y} R^\Gamma \brk{y, z} R^\Gamma \brk{z, x}} \leq  \Norm{\mathscr K_\Gamma}_{L^\infty \brk{\Delta_{x, y, z}}} \norm{\Delta_{x,y,z}}.
    \end{equation*}
\end{proposition}

Here we have set 
\begin{equation}
\label{eq_gigol5jahqu5Ou6ahmidaere}
 \Norm{\mathscr K_\Gamma}_{L^\infty \brk{\Omega}}
 = \sup_{x \in \Omega} \sup_{v, w \in F} \frac{\norm{\mathscr{K} \brk{x}\sqb{v, w}}}{\norm{v}\norm{w}}.
\end{equation}
Since \(\mathscr{K}\) is skew-symmetric, we have by an orthogonalisation argument
\begin{equation}
\label{eq_ij7koolaepheiphahQu2choo}
 \norm{\mathscr{K} \brk{x}\sqb{v, w}} \le  \Norm{\mathscr K_\Gamma}_{L^\infty \brk{\Omega}} \norm{v \wedge w},
\end{equation}
with \(\norm{v \wedge w} = \sqrt{\norm{v}^2 \norm{w}^2 - \inpr{v}{w}^2}\).

The norms on \(\Lin {F}{F}\) in \cref{estimate holonomy triangle} and \eqref{eq_gigol5jahqu5Ou6ahmidaere} have to be the same and can be any norm that is invariant under left- and right-composition with isometries of \(O \brk{F}\). 
We will use them for the operator norm; they also hold for the Hilbert--Schmidt norm.

In the abelian case where \(\mathscr{K}_\Gamma = \dd \Gamma\), \cref{estimate holonomy triangle} follows from Stokes' theorem~\citelist{\cite{Nguyen_VanSchaftingen_2020}*{Lemma 2.1}\cite{Chanillo_VanSchaftingen_2020}}.

\begin{proof}[Proof of \cref{estimate holonomy triangle}]
For \(s \in [0,1]\), we set
\begin{align*}    
    x_s &= x, &
	y_s &= s y + (1-s)x&
	&\text{and}
	&z_s = sz + (1-s)x.
\end{align*}
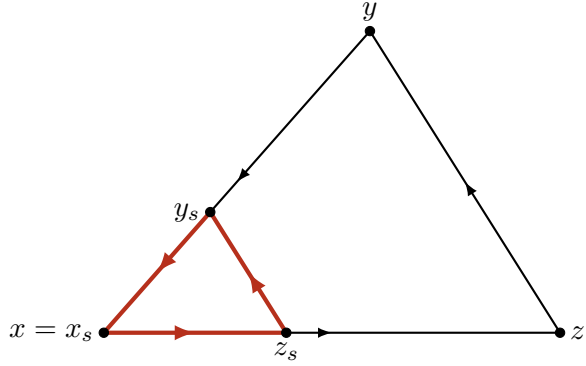
\begin{figure}
    \centering
    \colorlet{PathColor}{BrickRed}
    \begin{tikzpicture}[scale= 1]
    \draw[middlearrow={latex}, thick] (0,0) coordinate (x) --   (6,0) coordinate (z);
    \draw[middlearrow={latex}, thick] (z) -- (3.5,4) coordinate (y);
    \draw[middlearrow={latex}, thick] (y) -- (x);

    \draw[PathColor, middlearrow={latex}, ultra thick] (x) -- ($4/10*(z)+ 6/10*(x)$) coordinate (zr);
    \draw[PathColor, ultra thick, middlearrow={latex}] (zr) -- ($4/10*(y)+ 6/10*(x)$) coordinate (yr) ;
    \draw[PathColor, middlearrow={latex}, ultra thick] (yr) -- (x);

    \filldraw (x) circle (1.75pt) node[left] {$x = x_s$};
    \filldraw (y) circle (1.75pt) node[above] {$y$};
    \filldraw (z) circle (1.75pt) node[right] {$z$};

    \filldraw (yr) circle (1.75pt) node[left] {$y_s$};
    \filldraw (zr) circle (1.75pt) node[below] {$z_s$};

    \end{tikzpicture}
    \caption{The intermediate points \(x_s, y_s, z_s\) and the transport performed by \(W_s = R^\Gamma \brk{x_s, y_s} R^\Gamma \brk{y_s, z_s} R^\Gamma \brk{z_s, x_s}\).}
    \label{fig:triangle holonomie}
\end{figure}
We note that 
\begin{align*}
	y_s - z_s &= s \brk{y - z}, & y_s - x_s &= s \brk{y-x},  & z_s - x_s & = s\brk{z-x},\\
	\tfrac{\dd}{\dd s} x_s &= 0, & \tfrac{\dd}{\dd s} y_s, &= y - x, &\tfrac{\dd}{\dd s} z_s &= z - x.
\end{align*}
We define 
\[
 W_s = R^\Gamma \brk{x_s, y_s} R^\Gamma \brk{y_s, z_s} R^\Gamma \brk{z_s, x_s},
\]
and thus, 
\begin{equation*}
\begin{split}
 &\frac{\dd}{\dd s} W_s\\
 &\; =  \brk[\Big]{\frac{\dd}{\dd s} R^\Gamma\brk{x_s, y_s} - R^\Gamma \brk{x_s, y_s} \Gamma(y_s)\sqb[\big]{\tfrac{\dd y_s}{\dd s}}}
 R^\Gamma \brk{y_s, z_s} R^\Gamma \brk{z_s, x_s}\\
  & \quad + R^\Gamma \brk{x_s, y_s} \brk[\Big]{\frac{\dd}{\dd s} R^\Gamma\brk{y_s, z_s} + \Gamma(y_s)\sqb[\big]{\tfrac{\dd y_s}{\dd s}} R^\Gamma \brk{y_s, z_s}- R^\Gamma \brk{y_s, z_s} \Gamma(z_s)\sqb[\big]{\tfrac{\dd z_s}{\dd s}}} R^\Gamma \brk{z_s, x_s}\\
  &\quad + R^\Gamma \brk{x_s, y_s} R^\Gamma \brk{y_s, z_s}\brk[\Big]{\frac{\dd}{\dd s} R^\Gamma\brk{z_s, x_s} + \Gamma(z_s)\sqb[\big]{\tfrac{\dd z_s}{\dd s}} R^\Gamma \brk{z_s, x_s}},
\end{split}
\end{equation*}
and, since \(\tfrac{\dd x_s}{\dd s} = 0\), we have
\begin{equation}
\label{eq_ohshu2aiqu1ohth5AiL5ae7o}
\begin{split}
 \norm[\bigg]{\frac{\dd}{\dd s} W_s}
 \le  & \norm[\bigg]{\frac{\dd}{\dd s} R^\Gamma\brk{x_s, y_s} + \Gamma(x_s)\sqb[\big]{\tfrac{\dd x_s}{\dd s}} R^\Gamma \brk{x_s, y_s}- R^\Gamma \brk{x_s, y_s} \Gamma(y_s)\sqb[\big]{\tfrac{\dd y_s}{\dd s}}}\\
  & + \norm[\bigg]{\frac{\dd}{\dd s} R^\Gamma\brk{y_s, z_s} + \Gamma(y_s)\sqb[\big]{\tfrac{\dd y_s}{\dd s}} R^\Gamma \brk{y_s, z_s}- R^\Gamma \brk{y_s, z_s} \Gamma(z_s)\sqb[\big]{\tfrac{\dd z_s}{\dd s}}}\\
  &+ \norm[\bigg]{\frac{\dd}{\dd s} R^\Gamma\brk{z_s, x_s} + \Gamma(z_s)\sqb[\big]{\tfrac{\dd z_s}{\dd s}} R^\Gamma \brk{z_s, x_s}- R^\Gamma \brk{z_s, x_s} \Gamma(x_s)\sqb[\big]{\tfrac{\dd x_s}{\dd s}}}.
\end{split}
\end{equation}
Since \(x_s - y_s\), \(\tfrac{\dd}{\dd s} x_s\) and \(\tfrac{\dd}{\dd s} y_s\) are colinear,
it follows then from \eqref{eq_uyooGuezaik5icei3UCaeh3i} that the first term on the right-hand side of \eqref{eq_ohshu2aiqu1ohth5AiL5ae7o} vanishes; the last one vanishes similarly.
It remains thus to estimate the middle term: by \eqref{eq_uyooGuezaik5icei3UCaeh3i} we have 
\[
\begin{split}
  \norm[\bigg]{\frac{\dd}{\dd s} W_s}
  &\le \int_{0}^{1} \norm{\mathscr{K}_\Gamma(\brk{1-t}sy + t s z + (1 -s) x)[s \brk{y - z}, \brk{1-t}y + tz - x]} \dd t\\
  & = s \int_{0}^{1} \norm{\mathscr{K}_\Gamma(\brk{1-t}sy + t s z + (1 -s) x)[y - z,y - x]} \dd t.
\end{split}
\]
By integration and \eqref{eq_ij7koolaepheiphahQu2choo}, we deduce that 
\begin{equation}
 \begin{split}
  &\norm[\big]{\id_F - R^\Gamma \brk{x, y} R^\Gamma \brk{y, z}R^\Gamma \brk{z, x}}\\
  &\qquad \le \int_0^1 s \int_{0}^{1} \norm{\mathscr{K}_\Gamma(\brk{1-t}sy + t s z + (1 -s) x)[y - z,y - x]} \dd t \dd s\\
  &\qquad \le \Norm{\mathscr{K}_\Gamma}_{L^\infty \brk{\Delta_{x, y, z}}} \norm{\Delta_{x, y, z}}
\end{split}
\end{equation}
which proves the claim.
\end{proof}

We shall also need the following estimation for parallel transport along two paths forming a triangle;
\begin{lemma} \label{traces: prop2.1 lemma}
	If \(n\geq 1\), \(U\in  C^1(\smash{\overline{\R}}^{n + 1}_+, F)\) and \(\Gamma\in C^1(\smash{\overline{\R}}_+^{n+1}, \Lin{\R^{n+1}}{\mathfrak{o}(F)})\), then for every \(x\), \(y\), \(z \in \smash{\overline{\R}}^{n + 1}_+\), we have
	\begin{equation*}
	\begin{split}
		\norm{U\brk{x}- R^\Gamma\brk{x, y} U\brk{y}} &\leq \norm{U\brk{x} - R^\Gamma(x,z) U(z)} + \norm{U\brk{y}- R^\Gamma(y,z)U(z)}\\
		&\qquad + \norm{U(z)} \min\{2, \spNorm{\mathscr K_\Gamma }{\infty} \norm{\Delta_{x,y,z}} \}.
		\end{split}
	\end{equation*}
\end{lemma}

In the abelian case, \cref{traces: prop2.1 lemma} follows essentially from the Stokes formula for circulation integrals \cite{Nguyen_VanSchaftingen_2020}*{Lemma 2.3}.

\begin{proof}[Proof of \cref{traces: prop2.1 lemma}]
	Since \(R^\Gamma(v,w) \in O(F)\) is an isometry for all \(v,w\in \smash{\overline{\R}}^{n + 1}_+\),
	we have 
	\begin{equation}
	 \begin{split}
	    &\norm{U\brk{x} - R^\Gamma\brk{x, y}U\brk{y}}\\
	    &\qquad\le \norm{U\brk{x} - R^{\Gamma} \brk{x, z} U \brk{z}} + \norm{R^\Gamma \brk{x, z} U(z) - R^\Gamma \brk{x, y} R^\Gamma\brk{y, z}U(z)}\\
	    &\qquad \qquad +  \norm{R^\Gamma\brk{x, y} R^\Gamma\brk{y, z}U(z)
	    - R^\Gamma\brk{x, y} U\brk{y}}\\
	    &\qquad = \norm{U\brk{x} - R^\Gamma \brk{x, z} U \brk{z}} + \norm{U(z) - R^\Gamma\brk{z, x}R^\Gamma\brk{x, y} R^{\Gamma} \brk{y, z} U \brk{z}}\\
	    &\qquad \qquad + \norm{U\brk{x}-R^\Gamma\brk{x, z}U \brk{z}}.
	 \end{split}
	\end{equation}
We conclude by noting that 
\[
\norm{U(z) - R^\Gamma\brk{z, y}R^\Gamma\brk{y, x} R^{\Gamma} \brk{x, z} U \brk{z}} \le 2 \norm{U \brk{z}},
\]
and by \cref{estimate holonomy triangle}
\[
 \norm{U(z) - R^\Gamma\brk{z, y}R^\Gamma\brk{y, x} R^{\Gamma} \brk{x, z} U \brk{z}}\leq |U(z)|  \|\mathscr K_\Gamma \|_\infty |\Delta_{x,y,z}|. \qedhere
\]
\end{proof}

\section{Trace estimates on the half-space} 
\label{sect: traces}
We establish in this section \cref{prop traces} which, by taking \(s=1-\frac{1}{p}\), implies (i) in \cref{thm1}.
\begin{proposition}\label{prop traces}
Let \(n \geq 1\), \(0<s<1\) and \(1\leq p < +\infty\). There exists a positive constant \(C=C(n,s,p)>0\) depending only on \(n\), \(s\) and \(p\) such that for every \(\Gamma \in  C^1(\smash{\smash{\overline{\R}}_+^{n+1}}, \Lin{\R^{n+1}}{\mathfrak{o}(F)})\), \(\beta \ge \Norm{\mathscr K_\Gamma }_{\infty}\) and \(U\in  C^1_c(\smash{\smash{\overline{\R}}^{n + 1}_+},F)\)
\begin{equation} \label{prop traces eq1}
    \norm{ U (\cdot, 0)}^p_{W^{s,p}_{\Gamma_\shortparallel}(\R^{n},F)} \leq C \int_{\R^{n+1}_+} \frac{ \norm{\Diff_\Gamma  U(z)}^p + \beta^{p/2} \norm{ U(z)}^p}{z_{n + 1}^{1-(1-s)p}} \dd z
\end{equation}
and 
\begin{equation}\label{prop traces eq2}
\Norm{ U(\cdot , 0)}_{{L^p(\R^{n},F)}}^p
\leq C \Bigl( \int_{\R^{n + 1}_+} \frac{ \norm{\Diff_\Gamma  U(z)}^p}{z_{n +1}^{1-(1-s)p}} \dd z \Bigr)^{1-s} \brk[\bigg]{\int_{\R^{n +1}_+}\frac{\norm{U(z)}^p}{z_{n + 1}^{1-(1-s)p}}\dd z}^s.
\end{equation}
\end{proposition}

\begin{proof}
     For each \(x\), \(y \in \R^n\), setting \(z=( \frac{x+y}{2}, \norm{x - y})\), we have
    \begin{equation}\label{traces: proof prop 2.1 eq1}
        \min \{2, \|\mathscr K_\Gamma \|_\infty |\Delta_{x,y,z}| \} \leq \sqrt{2 \|\mathscr K_\Gamma \|_\infty |\Delta_{x,y,z}|} \le \norm{x - y} \beta^{1/2},
    \end{equation}
    and thus \cref{traces: prop2.1 lemma} with \eqref{traces: proof prop 2.1 eq1} yields, with the identification \(x \cong (x,0)\), \(y\cong (y,0)\),
    \begin{equation}\label{traces: proof prop 2.1 eq2}
    \begin{split}
        &|U\brk{x}-R^\Gamma \brk{x, y}U\brk{y} | \\
        &\qquad \leq |U\brk{x} - R^\Gamma (x,z)U(z) |
        + |U\brk{y}-R^\Gamma (y,z)U(z)| +\beta^{1/2} |U(z) | \norm{x - y} .
        \end{split}
    \end{equation}
    The gauge-covariant fundamental theorem of calculus, see \cref{FTC}, applied to \eqref{traces: proof prop 2.1 eq2} yields
    \begin{equation}
\label{prof prop 2.1 eq:3}
\begin{split}
| U\brk{x} - R^\Gamma (x, y)U\brk{y} |
& \leq \int_0^1 | \Diff_\Gamma  U(\brk{1-t}z + t x))[x - z]| \dd t\\
&\qquad + \int_0^1 | \Diff_\Gamma  U(\brk{1 - t} z + t y)[z - y] |\dd t\\
&\qquad + |U(z) | \norm{x - y} \beta^{1/2}.
        \end{split}
    \end{equation}
    Since \(z=( \frac{x+y}{2}, \norm{x - y})\), we have \(
      \norm{x - z} = \norm{y - z}
      \le C \norm{x - y}
    \) for some positive constant \(C>1\), and thus
    \begin{multline}\label{proof prop 2.1 eq:4}
        |U\brk{x} - R^\Gamma \brk{x, y} U\brk{y} | \\
        \leq C \norm{x - y}\brk[\bigg]{\int_0^1 |\Diff_\Gamma  U(\brk{1-t}z + t x) | \dd t + \int_0^1 |\Diff_\Gamma  U(\brk{1-t}z + t x) | \dd t + \beta^{1/2} | U(z) |  }.
    \end{multline}
Recalling that by H\"older's inequality, we have for any measurable function \(f\colon [0,1] \to [0, \infty]\)
 \begin{equation}\label{proof prop 2.1 inequality holder}
 \begin{split}
        \brk[\bigg]{\int_0^1 \norm{f \brk{t}}  \dd t }^p &\leq
        \brk[\bigg]{
        \int_0^1 \frac{1}{t^{1-s}}\dd t}^{p - 1}\int_0^1 t^{(1-s)(p-1)} \norm{f \brk{t}}^p \dd t\\
        &= \frac{1}{s^{p - 1}} \int_0^1 t^{(1-s)(p-1)} \norm{f \brk{t}}^p \dd t
\end{split}
    \end{equation}
and that \(\norm{a+b+c}^p\leq 3^{p-1}(|a|^p + |b|^p + |c|^p )\), raising \eqref{proof prop 2.1 eq:4} to the \(p\) and dividing by \(\norm{x - y}^{n-(1-s)p}\) yields
    \begin{equation}
\label{proof prop 2.1 eq:6}
\begin{split}
        &\frac{| U\brk{x} - R^\Gamma (x, y) U\brk{y}|^p}{\norm{x - y}^{n+sp}}\\
        &\quad \leq C \brk[\bigg]{ \int_0^1 t^{(1-s)(p-1)} \frac{|\Diff_\Gamma  U( (1-\frac{t}{2})x + \frac{t}{2}y , \frac{t}{2}\norm{x - y} ) |^p}{\norm{x - y}^{n-(1-s)p}} \dd t \\
        &\quad \qquad +\int_0^1 t^{(1-s)(p-1)} \frac{|\Diff_\Gamma  U( (1-\frac{t}{2})y + \frac{t}{2}x , \frac{t}{2}\norm{x - y} ) |^p}{\norm{x - y}^{n-(1-s)p}} \dd t
        + \beta^{p/2} \frac{|U(\frac{x+y}{2},\norm{x - y}) |^p}{\norm{x - y}^{n-(1-s)p}}}.
        \end{split}
    \end{equation}
    The integral on the left-hand side of \eqref{proof prop 2.1 eq:6} with respect to \(x\) and \(y\) will now be estimated by the integrals of the three terms on the right-hand side. For \(t\in (0,1)\), we perform the change of variable \(\eta = (1-\frac{t}{2})x + \frac{t}{2}y\) and \(\xi = t\brk{x-y}\) which has a Jacobian of \(t^{-d}\) and we obtain
    \begin{multline} \label{proof prop 2.1 eqf:0}
        t^{(1-s)(p-1)} \int_{\R^n} \int_{\R^n} \frac{|\Diff_\Gamma  U( (1-\frac{t}{2})x + \frac{t}{2}y , \frac{t}{2}\norm{x - y} ) |^p}{\norm{x - y}^{n-(1-s)p}} \dd x \dd y \\
        = \frac{1}{t^{1-s}} \int_{\R^n} \int_{\R^n} \frac{|\Diff_\Gamma  U (\eta,|\xi|)|^p}{|\xi|^{n-(1-s)p}}\dd \xi \dd \eta = \frac{|\SP^{n-1}|}{t^{1-s}} \int_{\R^n} \int_0^{+\infty} \frac{|\Diff_\Gamma  U(\eta , r) |^p}{r^{1-(1-s)p}} \dd r \dd \eta,
    \end{multline}
    where the last equality follows from integration of radial functions. Since \(0<s<1\), Tonelli's theorem and \eqref{proof prop 2.1 eqf:0} imply that there exists a constant \(C>0\) such that
    \begin{multline}\label{proof prop 2.1 eqf:1}
        \int_{\R^n} \int_{\R^n} \int_0^1 t^{(1-s)(p-1)} \frac{|\Diff_\Gamma  U( (1-\frac{t}{2})x + \frac{t}{2}y , \frac{t}{2}\norm{x - y} ) |^p}{\norm{x - y}^{n-(1-s)p}} \dd t\dd x \dd y \\
        \leq C \int_{\R^n} \int_0^{+\infty} \frac{|\Diff_\Gamma  U(x,r) |^p}{r^{1-(1-s)p}} \dd r \dd x
    \end{multline}
    holds, giving us the estimate for the first term of the right-hand side of \eqref{proof prop 2.1 eq:6}. For the second term of the right-hand side of \eqref{proof prop 2.1 eq:6} we obtain in a similar manner a constant \(C>0\) such that
    \begin{multline}\label{proof prop 2.1 eqf:2}
        \int_{\R^n} \int_{\R^n} \int_0^1 t^{(1-s)(p-1)} \frac{|\Diff_\Gamma  U( (1-\frac{t}{2})y + \frac{t}{2}x , \frac{t}{2}\norm{x - y} ) |^p}{\norm{x - y}^{n-(1-s)p}} \dd t\dd x \dd y \\
        \leq C \int_{\R^n} \int_0^{+\infty} \frac{|\Diff_\Gamma  U(x,r) |^p}{r^{1-(1-s)p}} \dd r \dd x.
    \end{multline}
    Finally for the last term of the right-hand side of \eqref{proof prop 2.1 eq:6}, a computation as in \eqref{proof prop 2.1 eqf:0} with the change of variable \(\eta = \frac{x+y}{2}\) and \(\xi=y-x\) as well as spherical coordinates yields
    \begin{equation}\label{proof prop 2.1 eqf:3}
        \int_{\R^n}\int_{\R^n} \frac{|U(\frac{x+y}{2}, \norm{x - y} ) |^p}{\norm{x - y}^{n-(1-s)p}} \dd x \dd y = \norm{\SP^{n-1}} \int_{\R^n}\int_0^{+\infty} \frac{|U(\eta,r)|^p}{r^{1-(1-s)p}} \dd r \dd \eta.
    \end{equation}
    Integrating both sides of \eqref{proof prop 2.1 eq:6} with respect to \(x\) and \(y\) and using the estimates \eqref{proof prop 2.1 eqf:1}, \eqref{proof prop 2.1 eqf:2} and \eqref{proof prop 2.1 eqf:3}, we reach
    \begin{equation*}
    | U (\cdot, 0)|^p_{W^{s,p}_{\Gamma_\shortparallel}(\R^{n},F)} \leq C \int_{\R^{n + 1}_+} \frac{ |\Diff_\Gamma  U(z)|^p + \beta^{p/2} | U(z) |^p}{z_{n + 1}^{1-(1-s)p}} \dd z
\end{equation*}
which is the announced estimate \eqref{prop traces eq1}. 
The proof of the first part of \cref{prop traces} is thus complete. 

We now work towards \eqref{prop traces eq2}. 
By the diamagnetic inequality (\cref{lemma: diamagnetic inequality}),
we have \(|\Diff |U| | \leq | \Diff_\Gamma  U |\) almost everywhere in \(\smash{\overline{\R}}^{n + 1}_+\) and
\begin{equation}\label{proof prop 2.1 p2 eq:1}
\begin{split}
    |U\brk{x} |
    &\leq \frac{1}{\lambda} \int_0^\lambda
    \brk[\bigg]{\int_0^t \norm{\Diff_\Gamma  U(x,\tau)} \dd \tau + \norm{U \brk{x, t}}}\dd t\\
    &\leq \int_{0}^\lambda  \norm{\Diff_\Gamma  U(x,t)}\dd t + \frac{1}{\lambda}\int_{0}^\lambda \norm{U(x,t)}\dd t .
\end{split}
\end{equation}
By H\"older's inequality, we have
\begin{equation} \label{proof prop 2.1 p2 eq:2}
\begin{split}
    \brk[\bigg]{\int_0^\lambda |\Diff_\Gamma  U(x,t) | \dd t }^p
    &\leq \brk[\bigg]{\int_0^\lambda \frac{1}{t^{1 - s\frac{p}{p-1}}} \dd t }^{p-1} \int_0^{+\infty} \frac{|\Diff_\Gamma  U(x,t) |^p}{t^{1-(1-s)p}} \dd t \\
    &= \brk[\Big]{\frac{p - 1}{sp}}^{p - 1}\lambda^{sp} \int_0^{+\infty} \frac{| \Diff_\Gamma  U (x,t) |^p}{t^{1 - (1-s)p}} \dd t
\end{split}
\end{equation}
as well as
\begin{equation}\label{proof prop 2.1 p2 eq:3}
  \brk[\bigg]{\frac{1}{\lambda} \int_{0}^\lambda |U(x,t)| \dd t }^p
  \leq \brk[\Big]{\frac{p - 1}{sp}}^{p - 1}  \frac{1}{\lambda^{(1-s)p}} \int_0^{+\infty} \frac{| U(x,t)|^p}{t^{1-(1-s)p}} \dd t.
\end{equation}
Injecting the estimates \eqref{proof prop 2.1 p2 eq:2} and \eqref{proof prop 2.1 p2 eq:3} into \eqref{proof prop 2.1 p2 eq:1} and integrating, we obtain a constant \(C>0\) such that
\begin{multline} \label{proof prop 2.1 p2 eq:4}
\int_{\R^n}|U\brk{x} |^p \dd x  \leq C \brk[\bigg]{ \lambda^{sp}  \int_{\R^{n + 1}_+} \frac{| \Diff_\Gamma  U (z) |^p}{z_{n + 1}^{1 - (1-s)p}} \dd z + \frac{1}{\lambda^{(1-s)p}} \int_{\R^{n + 1}_+} \frac{| U(z)|^p}{z_{n + 1}^{1-(1-s)p}} \dd z}.
\end{multline}
Finally, optimising \eqref{proof prop 2.1 p2 eq:4} with respect to \(\lambda > 0\) yields a constant \(C > 0\) such that 
\begin{equation*}
    \int_{\R^n}|U\brk{x} |^p \dd x \leq C \brk[\bigg]{\int_{\R^{n + 1}_+} \frac{| \Diff_\Gamma  U (z) |^p}{z_{n + 1}^{1 - (1-s)p}} \dd z}^{1-s}    \brk[\bigg]{\int_{\R^{n + 1}_+} \frac{| U(z)|^p}{z_{n + 1}^{1-(1-s)p}} \dd z}^s
\end{equation*}
which is \eqref{prop traces eq2} and thus concludes the proof of \cref{prop traces}.
\end{proof}

Using the same arguments, we obtain a local variant of \cref{prop traces} which will be used to transfer the result on a manifold with boundary in \cref{sect: geometric tools}. 
Here and in what follows, we denote by \(B(x,r)\) the \(n\)-dimensional ball centred in \(x\in \R^n\) with radius \(r>0\). When \(x=0\), we write \(B_r = B(0,r)\).
\begin{proposition}\label{prop: trace estimates local version}
    Let \(n\geq 1\), \(0<s<1\), \(1 \leq p < +\infty\) and \(0 < R < +\infty\).
    There is a positive constant \(C=C(n,s,p)>0\) such that for each \(\Gamma \in  C^1(\overline{B_R}\times [0,R], \Lin{\R^{n+1}}{\mathfrak {o}(F)})\), \( \beta \geq \Norm{\mathscr K_\Gamma }_{\infty}\) and \(U\in  C^1(B_R \times [0,R], F)\),
    \begin{equation*}
        \int_{B_R}\int_{B_R} \frac{\norm{U\brk{x} - R^\Gamma \brk{x, y}U\brk{y}}^p}{\norm{x - y}^{n+sp}}\dd x \dd y \\
        \leq C \int_{B_R \times (0,R)} \frac{\norm{\Diff_\Gamma  U(z) }^p+\beta^{p/2} \norm{U(z)}^p}{z_{n + 1}^{1-(1-s)p}} \dd z.
    \end{equation*}
\end{proposition}

\section{Extension to the half-space} \label{sect: extensions}
In this section, we establish \Cref{prop extension} which, by taking \(s=1-\frac{1}{p}\), implies (ii) in \cref{thm1}.
\begin{proposition}\label{prop extension}
    Let \(n \geq 1\), \(0<s<1\) and \(1\leq p < +\infty\).
    There exists a positive constant \(C=C(n,s,p)>0\) depending only on \(n\), \(s\) and \(p\) such that for every \(\Gamma \in   C^1(\smash{\smash{\overline{\R}}_+^{n+1}}, \Lin{\R^{n+1}}{\mathfrak{o}(F)})\), \(\beta \geq \Norm{\mathscr K_\Gamma}_\infty\) and \(u \in   C^1_c(\R^n, F)\) there exists
    \begin{equation*}
    U\in   C^1_c(\smash{\R^{n + 1}_+}, F) \cap   C^0_c(\smash{\smash{\overline{\R}}^{n + 1}_+}, F)
    \end{equation*}
    depending linearly on \(u\) and depending on \(\beta\) such that for all \(x \in \R^n\) we have \(U(x,0)= u\brk{x}\),
    \begin{equation}\label{prop extension eq:1}
        \int_{\R^{n+1}_+} \frac{\norm{\Diff_\Gamma U(z) }^p}{z_{n+1}^{1-(1-s)p}}\dd z\leq C ( \norm{u}^p_{W^{s,p}_{\Gamma_\shortparallel}(\R^n,F)} + \beta^{sp/2} \Norm{u}_{{L^p(\R^n,F)}}^p )
    \end{equation}
    and 
    \begin{equation}
        \label{prop extension eq:2} 
        \int_{\R^{n+1}_+}\frac{\norm{U(z)}^p}{z_{n+1}^{1-(1-s)p}}\dd z \leq \frac{C}{\beta^{(1-s)p/2}} \Norm{u}_{L^p(\R^n,F)}^p.
    \end{equation}
\end{proposition}

\begin{proof}[Proof of \cref{prop extension}]
We fix a function \(\varphi \in   C^\infty_c(\R^n)\) such that
    \begin{align}\label{proof extension def varphi}
        \int_{\R^n} \varphi\brk{y} \dd y &= 1 &
        &\text{and} &
        &\varphi = 0 &\text{ on } \R^n\setminus B(0,1)
    \end{align}
and we define for every \(t \in (0, \infty)\)
the function \(\varphi_t \colon \R^n \to \R\) for each \(y \in \R^n\) by
\begin{equation*}
        \varphi_{t} \brk{y} = \frac{1}{t^n} \varphi \brk[\Big]{\frac{y}{t}}.
\end{equation*}
We define the function \(V \colon \R^{n + 1}_+ \to F\)  for \(z=(z^\prime, z_{n+1}) \in \R^{n + 1}_+\) by
\begin{equation}\label{proof prop extension def u}
\begin{split}
    V(z)&= R^\Gamma(z, z')
    \brk[\big]{\brk{R^\Gamma \brk{z', \cdot} u} \ast \varphi_{z_{n + 1}}}\brk{z'}\\
    & = \frac{R^\Gamma(z, z')}{z_{n + 1}^{n} }\int_{\R^n}  R^\Gamma\brk{ z', x} u\brk{x} \,\varphi \brk[\Big]{\frac{z'-x}{z_{n + 1}}} \dd x
\end{split}
\end{equation}
and then \(U \colon \R^{n + 1}_+ \to F\) by
\begin{equation}
\label{eq_faimaenee4ob3yekahrai4Mi}
 U \brk{z} = \e^{-\beta z_{n+ 1}^2} V \brk{z}.
\end{equation}

We begin by showing that \(U\) takes the prescribed boundary value.
Given \(x \in \R^n \cong \R^n \times \set{0}\),
we have by Lebesgue's dominated convergence theorem
\begin{equation*}
\begin{split}
\lim_{\substack{z \to x\\ z_{n + 1} >0}} U(z) &=
    \lim_{\substack{z \to x\\ z_{n + 1} >0}} V(z)\\
    &= \lim_{\substack{z' \to x\\ z_{n + 1} \downarrow 0}}
     R^\Gamma(z, z')\int_{\R^n}  R^\Gamma( z', z' - z_{n + 1} y)\, u(z' - z_{n + 1} y)\, \varphi \brk{y} \dd y
     = u \brk{x},
     \end{split}
\end{equation*}
by \eqref{proof extension def varphi}, since \(R^\Gamma \brk{x, x} = \id_F\).

Since
\begin{equation}
\label{eq_OoX1wee7xi5eishahsh0lil7}
\int_{\R^n}
 \brk[\bigg]{\int_{B_{z_{n + 1}}\brk{z'}} \norm{u}^p}\dd z'
 \le C z_{n + 1}^{n} \int_{\R^n} \norm{u}^p,
\end{equation}
we also have
\begin{equation*}
 \int_{\R^n} \norm{V \brk{z', z_{n + 1}}}^p \dd z^\prime
 \le C \int_{\R^n} \norm{u}^p.
\end{equation*}
so that
\begin{equation}
\label{eq_Eechi4Ogah0weethahquip8p}
 \int_{\R^n} \norm{U \brk{z', z_{n + 1}}}^p \dd z'
 \le C \e^{-p \beta z_{n + 1}^2}\int_{\R^n} \norm{u}^p.
\end{equation}
and
\begin{equation*}
 \int_{\R^{n + 1}} \frac{\norm{U \brk{z', z_{n + 1}}}^p}{z_{n + 1}^{1 - \brk{1 - s}p}} \dd z
 \le C \int_0^\infty \frac{\e^{-p \beta t^2}}{t^{1 - \brk{1 - s}{p}}} \dd t\int_{\R^n} \norm{u}^p
 \le \frac{C'}{\beta^{\brk{1 - s}p/2}} \int_{\R^n} \norm{u}^p.
\end{equation*}

In order to estimate the derivative, we note that
\begin{equation}
\label{eq_doofu8goohahphieV5feeLu0}
\begin{split}
 \Diff_\Gamma V \brk{z}
 &=
 \frac{R^\Gamma(z, z')}{z_{n + 1}^{n + 1} }\int_{\R^n}  R^\Gamma\brk{ z', y}\, u\brk{y} \,\Phi \brk[\Big]{\frac{z'-y}{z_{n + 1}}} \dd y\\
 &\qquad +
 \frac{1}{z_{n + 1}^n}
 \int_{\R^n} \mathrm{H}^\Gamma \brk{z, y}\, u \brk{y} \, \varphi \brk[\Big]{\frac{z'-y}{z_{n + 1}}} \dd y.
\end{split}
\end{equation}
where \(\Phi \colon \R^{n} \to \Lin {\R^{n}}{\R}\) is defined for
\(y \in \R^n\) and \(v = \brk{v', v_{n + 1}} \in \smash{\R^{n + 1}_+}\) by
\[
  \Phi \brk{y}\sqb{v}
  = \Diff \varphi \brk{y}\sqb{v'} - \brk{\Diff \varphi \brk{y}\sqb{v^\prime} + n \varphi \brk{y}}v_{n + 1},
\]
and \(\mathrm{H}^\Gamma \colon \R^{n + 1}_+ \times \R^n \to \Lin{F}{\Lin{\R^n}{F}}\) is defined for every \(z \in \R^{n + 1}_+\) and \(x \in \R^n\) by
\[
 \mathrm{H}^\Gamma \brk{z, x}
 = \Diff_z \brk{\brk{R^\Gamma \brk{z, z'} R^\Gamma \brk{z', x}}}
 + \Gamma \brk{z} R^\Gamma \brk{z, z'} R^\Gamma \brk{z', x}.
\]
Rewriting the first term of the right-hand side of \eqref{eq_doofu8goohahphieV5feeLu0}, thanks to the identity\(\int_{\R^n} \Phi = 0\), as
\[
 \int_{\R^n}  R^\Gamma\brk{ z', x}\, u\brk{x} \,\Phi \brk[\Big]{\frac{z'-x}{z_{n + 1}}} \dd x
 = \int_{\R^n}  \brk{R^\Gamma\brk{ z', x} u\brk{x} - u \brk{z'}} \,\Phi \brk[\Big]{\frac{z'-x}{z_{n + 1}}} \dd x,
\]
we estimate it as
\begin{equation}
\label{eq_Cheithee4engifohNgohzuid}
\begin{split}
 &\norm[\bigg]{\frac{R^\Gamma(z, z')}{z_{n + 1}^{n + 1} }\int_{\R^n}  R^\Gamma\brk{ z', x}\,  u\brk{x} \,\Phi \brk[\Big]{\frac{z'-x}{z_{n + 1}}} \dd x}\\
 &\qquad \le \frac{1}{z_{n + 1}^{n + 1}} \int_{B_{z_{n + 1}} \brk{z'}} \norm{u\brk{x} - R^\Gamma\brk{x, z'}u\brk{z'}} \dd x.
 \end{split}
\end{equation}
Next, it follows from \cref{lemma: formula for derivative of Pt wrt parameter} and \eqref{eq_uyooGuezaik5icei3UCaeh3i} that
\[
\begin{split}
  \norm{\mathrm{H}^\Gamma \brk{z, y}\sqb{v}}
  &\le \int_{0}^1 \norm{\mathscr{K}_\Gamma \brk{z', tz_n} \sqb{\brk{0, z_n}, \brk{v', t v_n}}} \dd t\\
  &\qquad + \int_{0}^1 \norm{\mathscr{K}_\Gamma \brk{\brk{1 - t}y + t z'} \sqb{\brk{z'-y, 0}, \brk{v', 0}}} \dd t\\
  &\le C \beta \norm{y - z} \norm{v},
\end{split}
\]
and thus, in view of \eqref{eq_OoX1wee7xi5eishahsh0lil7},
\begin{equation}
\label{eq_Udeewahph6IeGh5ochee4ik1}
\norm[\bigg]{
 \frac{1}{z_{n + 1}^n}
 \int_{\R^n} \mathrm{H}^\Gamma \brk{z, x} u \brk{x} \, \varphi \brk[\Big]{\frac{z'-x}{z_{n + 1}}} \dd x}
 \le \frac{C \beta}{z_{n + 1}^{n - 1}}
 \int_{B_{z_{n + 1}}\brk{z'}} \norm{u \brk{x}} \dd x.
\end{equation}
Combining \eqref{eq_doofu8goohahphieV5feeLu0}, \eqref{eq_Cheithee4engifohNgohzuid} and \eqref{eq_Udeewahph6IeGh5ochee4ik1} with Jensen’s inequality, we get
\[
\norm{
 \Diff_\Gamma V \brk{z}}^p
 \le C
 \brk[\bigg]{\frac{1}{z_{n + 1}^{n + p}}
 \int_{B_{z_{n + 1}} \brk{z'}} \norm{u\brk{x} - R^\Gamma\brk{x, z'}u\brk{z'}}^p \dd x +
  \frac{\beta^p}{z_{n + 1}^{n - p}}
 \int_{B_{z_{n + 1}}\brk{z'}} \norm{u \brk{x}}^p \dd x
 }
\]
and thus, by \eqref{eq_OoX1wee7xi5eishahsh0lil7},
\begin{equation}
\label{eq_sua8koh2eXaiTh2zaiW0Achu}
\begin{split}
&\int_{\R^n}
 \frac{\norm{
 \Diff_\Gamma V \brk{z', z_{n + 1}}}^p}{z_{n+1}^{1 - \brk{1 - s}p}} \dd z'\\
&\qquad \le C \brk[\bigg]{\int_{\R^n}
 \int_{B_{z_{n + 1}} \brk{z'}} \frac{\norm{u\brk{x} - R^\Gamma\brk{ x, z'}u\brk{z'}}^p}{z_{n + 1}^{1 + n + sp}} \dd y \dd z'
 + \frac{\beta^p}{z_{n +1}^{1 - \brk{2 - s}p}}\int_{\R^n} \norm{u}^p}.
\end{split}
\end{equation}
By \eqref{eq_faimaenee4ob3yekahrai4Mi}, we have
\begin{equation}
  \norm{\Diff_\Gamma U \brk{z',z_{n + 1}}}
  \le \e^{-\beta z_{n + 1}^2} \brk{\norm{\Diff_\Gamma V \brk{z',z_{n + 1}}}
  + 2 \beta z_{n + 1}  \norm{V \brk{z', z_{n + 1}}}},
\end{equation}
and thus by \eqref{eq_sua8koh2eXaiTh2zaiW0Achu} and \eqref{eq_Eechi4Ogah0weethahquip8p}, we have
\begin{equation}
\begin{split}
  &\int_{\R^{n + 1}_+}
  \frac{\norm{\Diff_{\Gamma} U \brk{z', z_{n + 1}}}}{z_{n + 1}^{1 - \brk{1 - s}p}} \dd z\\
  &\le C
  \int_0^\infty \e^{-p \beta z_{n + 1}^2}
  \brk[\bigg]{
  \int_{\R^n}  \frac{\norm{\Diff_{\Gamma} V \brk{z', z_{n + 1}}}^p}{z_{n + 1}^{1 - \brk{1 - s}p}}
  +  \frac{\beta^p \norm{U \brk{z', z_{n + 1}}}^p}{z_{n + 1}^{1 - \brk{2 - s}p}} \dd z'}\dd z_{n + 1}\\
  &\le C'
  \int_0^\infty \e^{-p \beta t^2}
  \brk[\bigg]{ \int_{\R^n} \int_{B_{t} \brk{y}} \frac{\norm{u\brk{x} - R^\Gamma\brk{ x, y}u\brk{y}}^p}{t^{1 + n + sp}} \dd x \dd y
  +  \int_{\R^n} \frac{\beta^p \norm{u}^p}{t^{1 - \brk{2 - s}p}}}\dd t.
\end{split}
\end{equation}
In order to conclude, we note that
\[
 \int_{\norm{x - z'}}^\infty \frac{\e^{-p \beta t^2}}{t^{1 + n + sp}} \dd t
 \le \frac{C}{\norm{x - z'}^{n + sp}}
\]
and
\[
 \int_0^\infty \frac{\e^{-p \beta t^2}}{t^{1 - \brk{2 - s}p}} \dd t
 \le
 \frac{C}{
\beta^{\brk{1 - \frac{s}{2}}p}}. \qedhere
\]
\end{proof}

As for \cref{prop traces}, we have a local variant of \cref{prop extension} which will be used in \cref{sect: geometric tools} to transfer to result to manifolds with boundary.
\begin{proposition}\label{prop: extension local version}
    Let \(n\geq 1\), \(0<s<1\), \(1\leq p < +\infty\) and \(0 < R < +\infty\). If \(\Gamma\in   C^1( \overline{B_R}\times [0,R], \Lin{\R^{n+1}}{\mathfrak {o}(F)})\) and \(\Norm{\mathscr K_\Gamma}_{L^\infty(B_{2R}\times [0,R])} + \frac{1}{R^2} \leq \beta < +\infty\), then there exists a positive constant \(C=C(n,s,p)>0\) such that for all \(u \in   C^1(B_{2R},F)\) there exists 
    \begin{equation*}
        U \in   C^1( B_R \times (0,R), F) \cap   C^0( \overline{B_R}\times[0,R],F)
    \end{equation*}
    depending linearly on \(u\) such that \(U(\cdot ,0)=u(\cdot)\) on \(B_R\),
    \begin{multline*}
        \int_{B_R} \int_0^R \frac{\norm{\Diff_\Gamma U(z)}^p}{z_{n + 1}^{1-(1-s)p}} \dd z \\ 
        \leq C \Bigl(\int_{B_{2R}}\int_{B_{2R}} \frac{\norm{u\brk{x}- R^\Gamma\brk{x, y}u\brk{y} }^p}{\norm{x-y}^{n+sp}} \dd x \dd y + \beta^{sp/2} \int_{B_{2R}} \norm{U\brk{x}}^p \dd x \Bigr)
    \end{multline*}
    and
    \begin{equation*}
        \int_{B_R} \int_0^R \frac{\norm{U(z)}^p}{z_{n+1}^{1-(1-s)p}} \dd z \leq \frac{C}{\beta^{(1-s)p/2}} \int_{B_{2R}} \norm{u\brk{x}}^p \dd x.
    \end{equation*}
\end{proposition}

\section{Characterisation of trace spaces on the half-space}\label{section: density}
We establish in this section the density of smooth and compactly supported functions in the gauge-covariant Sobolev spaces on \(\smash{\overline{\R}}^{n + 1}_+\). Together with \cref{thm1}, \cref{density in WA} and \cref{density in Wsp A'} imply \cref{trace thm on Rn} in the introduction.

For \( \alpha \in \R\), we define the first order weighted gauge-covariant Sobolev space
\begin{equation*}
	W^{1,p}_{\Gamma , \alpha}(\smash{\smash{\overline{\R}}^{n + 1}_+}, F) = \set[\Big]{U \in W^{1,1}_{\mathrm{loc}}(\smash{\smash{\overline{\R}}^{n + 1}_+}, F) \st \Norm{U}_{W^{1,p}_{\Gamma, \alpha}(\smash{\smash{\overline{\R}}^{n + 1}_+},F)} < +\infty}
\end{equation*}
where
\begin{equation}
	\label{eq: Weighted Norm}
	\Norm{U}_{W^{1,p}_{\Gamma, \alpha}(\smash{\smash{\overline{\R}}^{n + 1}_+},F)} = \Bigl(\int_{\R^{n + 1}_+} \bigl(\norm{U(z)}^p + \norm{\Diff U (z)}^p \bigr) z_{n + 1}^{\alpha} \dd z\Bigr)^{1/p}.
\end{equation} 
With this norm, the weighted space is a Banach space. The energies on the right-hand side of \cref{prop traces} and the left-hand side of \cref{prop extension} are associated to the weighted space for \(\alpha = (1-s)p-1\).

The density results presented in \cref{density in WA,density in Wsp A'} were established by Nguyen and the first author in the scope of magnetic Sobolev spaces~\cite{Nguyen_VanSchaftingen_2020}*{\S 4} and directly carry over to the general case.

\begin{lemma}\label{density in WA}
Let \(1 \leq p < +\infty\) and \(-1 < \alpha < p-1\). If \(\Gamma\in L_{\mathrm{loc}}^\infty(\smash{\R^{n + 1}_+}, \Lin{\R^{n+1}}{\mathfrak {o}(F)})\), then \(C^\infty_c( \smash{\smash{\overline{\R}}^{n + 1}_+}, F)\) is dense in \(W_{\Gamma,\alpha}^{1,p}(\smash{\overline{\R}}^{n+1}_+,F)\).
\end{lemma}
\begin{proof}
Let $\chi \in C^\infty_c(\R^{n+1},\R)$ such that $0\leq \chi \leq 1$ and $\chi\brk{x}=1$ for all $|x|<1$. For all $\lambda > 0$, we define \(
	\chi_\lambda \colon \R^{n+1} \to \R\),
	by setting for \(x \in \R^{n + 1}\), \(\chi_\lambda\brk{x} = \chi(x/\lambda)\).
We then compute, for $U\in W^{1,p}_{\Gamma, \alpha}(\smash{\smash{\overline{\R}}^{n + 1}_+},F)$,
\begin{equation*}
	\Diff_\Gamma(U - \chi_\lambda U) = -U\Diff\chi_\lambda +(1-\chi_\lambda)\Diff_\Gamma U.
\end{equation*}
Therefore, we have 
\begin{multline*}
	\|U - \chi_\lambda U\|^p_{W^{1,p}_{\Gamma, \alpha}(\smash{\overline{\R}}^{n+1}_+,F)} \\
	\leq C \int_{\R^{n+1}} \Bigl((1-\chi_\lambda\brk{x})^p ( \norm{U \brk{x}}^p + | \Diff_\Gamma U\brk{x}|^p) + |\Diff\chi_\lambda\brk{x}|^p \norm{U \brk{x}}^p\Bigr) x_{n + 1}^\alpha \dd x \to 0
\end{multline*}
as $\lambda \to 0$ since $|\Diff\chi_\lambda| \leq C/\lambda$ for some constant $C> 0$. It follows that any function in $\smash{W^{1,p}_{\Gamma,\alpha}}(\smash{\R^{n + 1}_+},F)$ can be approximated by functions in $\smash{W^{1,p}_{\Gamma,\alpha}}(\R^{n+1},F)$ whose support is bounded.

Since \(-1 < \alpha < p-1\), any function in \(\smash{W^{1,p}_{0,\alpha}}(\smash{\smash{\overline{\R}}^{n+1}}, F)\) can be extended by even reflection to the entire space \(\smash{\R^{n+1}}\) and can be approximated in \(\smash{W^{1,p}_{0,\alpha}}(\smash{\smash{\overline{\R}}^{n+1}},F)\) by smooth functions with bounded support \citelist{\cite{Miller_1982}*{Lemma 2.4}\cite{Muckenhoupt_Wheeden_1978}*{Lemma 8}\cite{Tureson_2000}*{Cor.\ 2.1.5}}.
In particular, $C_c^\infty(\smash{\smash{\overline{\R}}^{n + 1}_+},F)$ is dense in the trivial gauge-covariant Sobolev space \(\smash{W^{1,p}_{0,\alpha}}(\smash{\smash{\overline{\R}}^{n + 1}_+},F)\) by restricting smooth and compactly supported functions on \(\smash{\R^{n + 1}}\) to the half space.

Since $\Gamma$ is locally bounded on \(\smash{\R^{n+1}_+}\), any function $U\in \smash{W^{1,p}_{\Gamma,\alpha}}(\smash{\smash{\overline{\R}}^{n+1}},F)$ whose support is bounded in \(\R^{n + 1}_+\) is also in $\smash{W^{1,p}_{0,\alpha}}(\smash{\smash{\overline{\R}}^{n+1}},F)$.
In fact, for such a function $U\in W^{1,p}_{\Gamma, \alpha}(\smash{\smash{\overline{\R}}^{n + 1}_+},F)$ with bounded support, we have
\begin{equation*}
	 \Norm{U}_{\smash{W^{1,p}_{0,\alpha}}(\smash{\smash{\overline{\R}}^{n+1}},F)} \leq (1+ \Norm{\Gamma}_{L^\infty(\Supp U)}) \Norm{U}_{\smash{W^{1,p}_{\Gamma,\alpha}}(\smash{\smash{\overline{\R}}^{n+1}},F)}.
\end{equation*}
Since the space $ C^\infty_c(\smash{\smash{\overline{\R}}^{n + 1}_+},F)$ is dense $\smash{W^{1,p}_{0,\alpha}}(\smash{\smash{\overline{\R}}^{n+1}},F)$, a function $U\in \smash{W^{1,p}_{\Gamma,\alpha}}(\smash{\smash{\overline{\R}}^{n+1}},F)$ with bounded support can be approximated in $\smash{W^{1,p}_{0,\alpha}}(\smash{\smash{\overline{\R}}^{n+1}},F)$ by $ C^\infty_c(\smash{\smash{\overline{\R}}^{n + 1}_+},F)$ functions. Indeed, for a mapping $U\in \smash{W^{1,p}_{\Gamma,\alpha}}(\smash{\smash{\overline{\R}}^{n+1}},F)$ with bounded support, we have
\begin{equation*}
	\Norm{U}_{\smash{W^{1,p}_{\Gamma,\alpha}}(\smash{\smash{\overline{\R}}^{n+1}},F)} \leq (1+\Norm{\Gamma}_{L^\infty(\Supp U)}) \Norm{U}_{\smash{W^{1,p}_{0,\alpha}}(\smash{\smash{\overline{\R}}^{n+1}},F)}.
\end{equation*}
It follows that $\smash{W^{1,p}_{\Gamma,\alpha}}(\smash{\smash{\overline{\R}}^{n+1}},F)$ functions with bounded support can be approximated in the space $\smash{W^{1,p}_{\Gamma,\alpha}}(\smash{\smash{\overline{\R}}^{n+1}},F)$ by smooth and compactly supported functions from the space $C_c^\infty(\smash{\smash{\overline{\R}}^{n + 1}_+},F)$. The conclusion follows by a diagonal argument.
\end{proof}

Regarding the fractional gauge-covariant Sobolev space \(\smash{W^{s,p}_{\Gamma_{\shortparallel}}}(\R^n,F)\), we have
\begin{lemma}\label{density in Wsp A'}
Let \(0<s<1\) and \(1 \leq p < +\infty\). If \(\Gamma_\shortparallel \in C^{0}(\R^n,\Lin{\R^n}{\mathfrak {o}(F)})\), then \( C^\infty_c( \R^n,F)\) is dense in \(W^{s,p}_{\Gamma_\shortparallel}(\R^n,F)\).
\end{lemma}
\begin{proof}
Consider $\chi \in C_c^\infty(\R^n,\R)$ such that $0\leq \chi \leq 1$ and $\chi =1$ on the unit ball $B_1$. For $\lambda > 0$, we define \(
\chi_\lambda \colon \R^n \to \R\)
by \(\chi_\lambda\brk{x} = \chi(x/\lambda)\).
Let $u\in \smash{W^{s,p}_{\Gamma_{\shortparallel}}}(\R^n,F)$. Then, for all $\lambda > 0$ and for all $x$, $y\in\R^n$, we have
\begin{multline}\label{densite wsp eq1}
(1-\chi_\lambda\brk{x})u\brk{x} - (1-\chi_\lambda\brk{y})R^{\Gamma_\shortparallel}\brk{x, y}u\brk{y} \\
=\biggl(1- \frac{\chi_\lambda\brk{x}+\chi_\lambda\brk{y}}{2}\biggr)
\brk{u\brk{x}-R^{\Gamma_\shortparallel}\brk{x, y}u\brk{y}}
- \frac{\chi_\lambda\brk{x}-\chi_\lambda\brk{y}}{2}
\brk{u\brk{x} + R^{\Gamma_\shortparallel}\brk{x, y}u\brk{y}}.
\end{multline}
Furthermore, for all $y\in\R^n$ and $\lambda > 0$, we have
\begin{equation}\label{densite wsp eq2}
\int_{\R^n} \frac{| \chi_\lambda\brk{x} - \chi_\lambda\brk{y}|^p}{\norm{x - y}^{n+sp}}\dd x \leq \frac{C}{\lambda^{sp}}
\end{equation}
for some constant $C>0$. It follows from \eqref{densite wsp eq1} and \eqref{densite wsp eq2} that, for every $\lambda > 0$,
\begin{multline}\label{densite wsp eq3}
| \chi_\lambda u - u|^p_{\smash{W^{s,p}_{\Gamma_{\shortparallel}}}(\R^n,F)} \\
\leq C \int_{\R^n}\int_{\R^n} \biggl(1-\frac{\chi_\lambda\brk{x}+\chi_\lambda\brk{y}}{2} \biggr)^p \frac{| u\brk{x} - R^{\Gamma_\shortparallel}\brk{x, y} u\brk{y} |^p}{\norm{x - y}^{n+sp}}\dd y \dd x + C^\prime \int_{\R^n} \frac{\norm{u\brk{x}}^p}{\lambda^{sp}} \dd x,
\end{multline}
where $C$, $C^\prime > 0$ are two constants. We deduce from \eqref{densite wsp eq3} and Lebesgue's dominated convergence theorem that 
\begin{equation*}
|u- \chi_\lambda u|_{W^{s,p}_{{\Gamma_{\shortparallel}}}(\R^n,F)} \to 0, \quad\lambda \to +\infty
\end{equation*}
which implies that $\Norm{ u - \chi_\lambda u}_{\smash{W^{s,p}_{\Gamma_{\shortparallel}}}(\R^n,F)} \to 0$ as $\lambda \to +\infty$. In other words, $\smash{W^{s,p}_{\Gamma_{\shortparallel}}}(\R^n,F)$ functions with bounded support are dense in $\smash{W^{s,p}_{\Gamma_{\shortparallel}}}(\R^n,F)$.

Since ${\Gamma_{\shortparallel}}$ is locally bounded on $\R^n$, any function in $\smash{W^{s,p}_{\Gamma_{\shortparallel}}}(\R^n,F)$ whose support is bounded is also in the trivial fractional gauge-covariant Sobolev space $W^{s,p}_0(\R^n,F)=W^{s,p}(\R^n,F)$. Indeed, if $u \in L^p(\R^n,F)$ has bounded support, then, by the triangle inequality, we have
\begin{multline}\label{1234}
\int_{\R^n}\int_{\R^n} \norm[\bigg]{\frac{|u\brk{x} - R^{\Gamma_\shortparallel}\brk{x, y}u\brk{y}|^p}{\norm{x - y}^{n+sp}} - \frac{|u\brk{x}-u\brk{y}|^p}{\norm{x - y}^{n+sp}}} \dd x \dd y \\
\leq C \int_{\R^n} \int_{\R^n} \frac{|u\brk{x} -  R^{\Gamma_\shortparallel}\brk{x, y}u\brk{y}|^p}{\norm{x - y}^{n+sp}} \dd x \dd y \\
\leq \sup_{y\in \Supp u} \biggl(\int_{\R^n} \frac{|\id_F  - R^{\Gamma_\shortparallel}\brk{x, y}|^p}{\norm{x - y}^{n+sp}}\dd x\biggr)\int_{\R^n} | u \brk{y}|^p \dd y.
\end{multline}
Since the parallel transport is an isometry, we have in the operator norm
\begin{equation}\label{<=2}
\norm{\id_F - R^{\Gamma_\shortparallel}(x,y)} \leq
+ \norm{\id_F} + \norm{R^{\Gamma_\shortparallel}(x,y)}
= 
2.
\end{equation}
By \eqref{def: Pt}, \eqref{eq: def R}
and by the fundamental theorem of calculus, we have
\begin{equation}\label{101}
\begin{split}
\norm{\id_F - R^{\Gamma_\shortparallel}\brk{x, y}}
&= \biggl|\int_0^1 -{\Gamma_{\shortparallel}}(\brk{1 - t} x + t y)[y-x]R^{\Gamma_\shortparallel}(\brk{1 - t} x + t y,y) \dd t \biggr|\\
&\leq \| {\Gamma_{\shortparallel}} \|_{L^\infty([x, y])} |x-y|.
\end{split}
\end{equation}
It follows from \eqref{<=2} and \eqref{101} that
\begin{multline}\label{RTG}
 \sup_{y\in \Supp u}\int_{\R^n} \frac{ \norm{\id_F - R^{\Gamma_\shortparallel}\brk{x, y}}^p}{\norm{x - y}^{n+sp}}\dd x \\
 \leq C \sup_{y \in \Supp u}\int_{\R^n} \frac{\min\bigl\{1 , \norm{x - y}^p \|{\Gamma_{\shortparallel}}\|_{L^\infty(\Supp u)}^p\bigr\}}{\norm{x - y}^{n+sp}}\dd x < + \infty.
\end{multline}
Combining \eqref{1234} and \eqref{RTG}, we deduce that if $u\in L^p(\R^n, F)$ has compact support, then
\begin{equation*}
u \in \smash{W^{s,p}_{\Gamma_{\shortparallel}}}(\R^n ,F) \iff u \in W_0^{s,p}(\R^n,F).
\end{equation*}
Since the set $C^\infty(\R^n,F)$ is dense in $W^{s,p}(\R^n,F)$, see~\cite{Leoni_fractionnal_2023}*{Thm.\ 6.66}, any function $u\in \smash{W^{s,p}_{\Gamma_{\shortparallel}}}(\R^n,F)$ can be approximated in $W^{s,p}_0(\R^n,F)$ by functions with uniformly compact support. Since ${\Gamma_{\shortparallel}}$ is locally bounded, the above calculation shows that the approximating sequence also converges in $\smash{W^{s,p}_{\Gamma_{\shortparallel}}}(\R^n,F)$, and the conclusion follows by a diagonal argument.
\end{proof}

Combining \cref{prop traces,prop extension} together with \cref{density in WA,density in Wsp A'}, we obtain the non-abelian counterpart of \cite{Nguyen_VanSchaftingen_2020}*{Thm.\ 4.3} which is the extension of the trace theorem for weighted Sobolev spaces~\cite{Uspenskii_1961} (see also~\citelist{\cite{Leoni_fractionnal_2023}*{\S 9.1}\cite{Mironescu_Russ_2015}}) to the magnetic case.
If \(1 < p < +\infty\) and \(s = 1- \frac{1}{p}\), \cref{thm: weighted traces} is \cref{trace thm on Rn} in the introduction.
\begin{theorem}
	\label{thm: weighted traces}
	Let \(n \geq 1\), \(0 < s < 1\) and \(1\leq p < +\infty\). If \(\Gamma \in C^1(\smash{\smash{\overline{\R}}^{n+1}},\Lin{\R^{n+1}}{\mathfrak{o}(F)}\) and if \(\Norm{\mathscr{K}_\Gamma}_\infty \leq \beta\), then there exists a trace mapping \[\Tr\colon {W^{1,p}_{\Gamma,(1-s)p-1}}(\smash{\overline{\R}}^{n+1}, F) \to W^{s,p}_{\Gamma_\shortparallel}(\R^n, F)\]
	and a positive constant \(C = C(n,s,p)\) depending only on \(n\), \(s\) and \(p\) such that if \(U \in C^1_c(\smash{\overline{\R}}^{n+1}, F)\), then \(\Tr U = U(\cdot,0)\) and for every \(U \in \smash{W^{1,p}_{\Gamma,(1-s)p-1}}(\smash{\smash{\overline{\R}}^{n+1}}, F)\), if \(u = \Tr U\),
	\begin{equation*}
		\norm{u}^p_{W^{s,p}_{\Gamma_\shortparallel}(\R^n,F)} + \beta^{sp/2}\Norm{u}^p_{L^p(\R^n,F)} \leq C \int_{\R^{n+1}_+} \frac{\norm{\Diff_\Gamma U(z)}^p + \beta^{p/2} \norm{U(z)}^p}{z_{n + 1}^{1-(1-s)p}} \dd z.
	\end{equation*} 
	Furthermore, there exists a linear continuous extension mapping \[\Ext \colon W^{s,p}_{\Gamma_\shortparallel}(\R^n, F) \to \smash{W^{1,p}_{\Gamma,(1-s)p-1}(\smash{\smash{\overline{\R}}^{n+1}}, F)}\] and a positive constant \(C= C(n,s,p)\) depending only on \(n\), \(s\) and \(p\) such that \(\Tr \circ \Ext\) is the identity on \(W^{s,p}_{\Gamma_\shortparallel}(\R^n,F)\) and such that  for each \(u \in W^{s,p}_{\Gamma_\shortparallel}(\R^n, F)\), if \(U = \Ext u\),
	\begin{equation*}
		\int_{\R^{n+1}_+} \frac{\norm{\Diff_\Gamma U(z)}^p + \beta^{p/2}\norm{U(z)}^p}{z_{n + 1}^{1-(1-s)p}} \dd z \leq C \Bigl(\norm{u}^p_{W^{s,p}_{\Gamma_\shortparallel}(\R^n, F)} + \beta^{sp/2} \Norm{U}_{L^p(\R^n,F)} \Bigr).
	\end{equation*}
\end{theorem}

\section{Traces and extensions on a manifold}\label{sect: geometric tools}

In this section we combine the results of the previous sections to characterise the traces on the boundary of a Riemannian vector bundle.
Following the strategy of \cite{Nguyen_VanSchaftingen_2020}*{\S 6}, we consider local trivialising charts of \(\mathcal{E}\) on which the local results on the trace and extension (\cref{prop: trace estimates local version,prop: extension local version}) are applicable and glue them back together. Throughout this section, we consider a smooth and compact Riemannian manifold \(\mathcal{M}\) with boundary \(\partial \mathcal{M}\).

A smooth manifold \(\mathcal{E}\) together with a smooth map \(\pi_\mathcal{E} \colon \mathcal{E} \to \mathcal{M}\) is called a \emph{Riemannian vector bundle} above the smooth manifold \(\mathcal{M}\) with the Euclidean space \(F\) as model fibre whenever \(\mathcal{E}\) is locally of the form \(\mathcal{M}  \times F\) in the following sense.
First, for each \(x \in \mathcal{M}\), the \emph{fibre} \(\mathcal{E}_x = \pi_\mathcal{E}^{-1}(\{x\})\) above each \(x\) is a closed submanifold of \(\mathcal{E}\) and is endowed with a Euclidean  structure isometric to \(F\).
Next, there exists an open covering \(\{V_\alpha\}_{\alpha \in \mathbf{A}}\) of \(\mathcal{M}\) and smooth diffeomorphisms \(\{\varphi_\alpha \colon \pi_\mathcal{E}^{-1}(V_\alpha) \to V_\alpha \times F\}_{\alpha \in \mathbf{A}}\) such that any restriction to a fibre \(\varphi_{\alpha \vert_{\mathcal{E}_x}} \colon \mathcal{E}_x \to \{x\}\times F\) for \(x\in V_\alpha\) is an isometry.
Such a tuple \((V_\alpha, \varphi_\alpha)\) is called a \emph{local trivialisation} or \emph{bundle chart} of \(\mathcal{E}\).

Typical examples of Riemannian vector bundles include trivial bundles \(\mathcal E = \mathcal{M}\times F\) as well as the tangent bundle \(\mathcal E = T\mathcal{M}\) of a Riemannian manifold \(\mathcal{M}\) or the normal bundle of an embedded Riemannian manifold.

If \((V_\alpha, \varphi_\alpha)\) and \((V_\beta, \varphi_\beta)\) are two intersecting isometric trivialisations of a Riemannian bundle \(\mathcal{E}\), that is, \(V = V_\alpha \cap V_\beta \neq \varnothing\), one can consider the \emph{change of trivialisation} or \emph{change of gauge} \(\varphi_\beta \circ \varphi_\alpha^{-1}\) from \(V_\alpha\) to \(V_\beta\). The smooth, fibre-wise linear and isometric characteristics of the trivialisations \(\varphi_\alpha\) and \(\varphi_\beta\) imply that there exists a unique smooth function \(\phi_{\beta,\alpha} \colon V \to O(F)\) such that for all \((x,v) \in V \times F\) 
\begin{equation*}
	\bigl(\varphi_\beta \circ \varphi_\alpha^{-1}\bigr)(x,v)  = (x, \phi_{\beta,\alpha}\brk{x} v).
\end{equation*}
The mapping \(\phi_{\beta,\alpha} \colon V \to O(F)\) is called the \emph{gauge transformation} from \((V_\alpha, \varphi_\alpha)\) to \((V_\beta, \varphi_\beta)\). In the particular case of the tangent bundle \(T\mathcal{M}\) of a Riemannian manifold \(\mathcal{M}\), the gauge transformations are the differentials of the change of charts of the manifold \(\mathcal{M}\). In settings of electromagnetism in quantum physics, these gauge transformations are phase shifts and take the form \(\e^{\ii \theta} \colon V \to U(1) \cong SO(2)\) for some smooth scalar function \(\theta\colon V \to \R\).

A \emph{connection} \(K\) on \(\mathcal{E}\) is given by a covariant derivative \(\Diff_K\), which is a differential operator such that for every local trivialisation \(\varphi \colon \pi_{\mathcal{E}}^{-1}(V) \to V \times F\) there is a unique section \(\Gamma \in  C (V, \Lin{T \mathcal{M}}{\mathfrak{o}(F)})\), called the \emph{connection form}, such that
one has 
\begin{equation}
\label{eq_Na2shee2oTienoh8ooShae2S}
  \varphi \brk{\Diff_K U} = \Diff_{\Gamma} \brk{\varphi \circ U} 
  = \Diff \brk{\varphi \circ U}
 + \Gamma \brk{\varphi \circ U}.
\end{equation}
If \(\varphi' \colon \pi_{\mathcal{E}}^{-1}(V) \to V \times F\) is another local trivialisation and \(\Gamma' \in  C (V, \Lin{T \mathcal{M}}{\mathfrak{o}(F)})\) is the associated connection form, we should have the gauge-invariance
\[
\begin{split}
 \Diff (\varphi' \circ U)
 + \Gamma' \brk{\varphi' \circ U}
 &= (\varphi'\circ \varphi^{-1})
 (\Diff (\varphi \circ U)
 +\Gamma \brk{\varphi \circ U})\\
 &= \Diff (\varphi' \circ U) + \bigl( (\varphi'\circ \varphi^{-1})\Gamma - \Diff (\varphi'\circ \varphi^{-1})\bigr)(\varphi'\circ \varphi^{-1})^{-1} \brk{\varphi' \circ U}.
\end{split}
\]
This will be the case provided
\[
  \Gamma' = (\varphi'\circ \varphi^{-1})\Gamma(\varphi'\circ \varphi^{-1})^{-1} - \Diff (\varphi'\circ \varphi^{-1})(\varphi'\circ \varphi^{-1})^{-1},
\]
that is, \eqref{eq: Gamma^prime} holds with the gauge transformation \(\phi  = \varphi'\circ \varphi^{-1}\).
The identity \eqref{eq_Na2shee2oTienoh8ooShae2S} still holds for 
\(U\) belonging to the gauge-covariant Sobolev space \(W^{1,p}_K(\mathcal{M}, \mathcal E)\) defined as the set of sections \(U \in W^{1, 1}_{\mathrm{loc}} (\mathcal{M}, \mathcal{E})\) such that
\[
 \int_{\mathcal{M}} \norm{\Diff_K U}^p + \norm{U}^p < +\infty.
\]
\begin{rmq}
When \(\mathcal E = \smash{\smash{\overline{\R}}^{n+1}} \times F\) is a trivial bundle and \(\Gamma\) represents the covariant derivative in a global trivialisation, we have 
\begin{equation*}
	W^{1,p}_K(\smash{\smash{\overline{\R}}^{n+1}}, \mathcal E) = W^{1,p}_\Gamma(\smash{\smash{\overline{\mathcal \R}}^{n+1}}, F)
\end{equation*} 
where the right-hand side is the space considered in the previous sections and defined by \eqref{def: W1p_Gamma}.
\end{rmq}
If \(\psi \colon W \subseteq \R^{n + 1}_+ \to \Omega \subset \mathcal{M}\) is a diffeomorphism and \(\Omega\) a local trivialisation,
the associated connection form \(\Gamma\) can be \emph{pulled-back} by \(\psi\) to obtain a connection form \(\psi^\ast \Gamma\) on \(W\) which is defined by~\cites{Sontz_bundle, JohnLee_2012, JeffreyLee_2009, Kobayashi_Nomizu_1963}  
\begin{equation}
	\label{def: pullback Gamma}
\psi^\ast\Gamma\brk{x}[\cdot]
= \Gamma\brk{\psi\brk{x}} \circ \Diff \psi\brk{x} [\cdot] \in \Lin{\R^{n+1}}{\mathfrak{o}(F)}
\end{equation}
for every \(x\in W\). 
Similarly, when \(U \in W^{1,p}_\Gamma(\Omega, F)\), one can pull-back the weak covariant derivative as
\begin{equation*}
\psi^\ast\brk{\Diff_\Gamma U}\brk{x} = \brk{\Diff_{\Gamma} U} \brk{\psi\brk{x}} \circ \Diff \psi\brk{x}.
\end{equation*} 
and obtain (\cref{lemma: pull-back of connection})
\begin{equation}
\label{eq_Ohjae8aenee6jo4aimusef0w}
  \Diff_{\psi^\ast \Gamma} (U \circ \psi)
  = \psi^\ast \brk{\Diff_\Gamma U}\brk{x}
\end{equation}
In view of \eqref{eq_kifaelaer1feiThohV1aesh2} and \eqref{eq_Ohjae8aenee6jo4aimusef0w}, the pull-back is compatible with the curvature: 
\begin{equation}
  \mathscr{K}_{\psi^\ast \Gamma}
  = \psi^\ast  \mathscr{K}_{\Gamma},
\end{equation}
where 
\[
 \brk{\psi^\ast  \mathscr{K}_{\Gamma}} (x)\sqb{v, w}
 = \mathscr{K}_{\Gamma} \brk{\psi\brk{x}}\sqb{\Diff\psi \brk{x}\sqb{v}, \Diff \psi \brk{x}\sqb{w}}.
\]
Also, since the pull-back acts on the domain side while gauge transformations work on the target side, pull-backs and gauge transformations commute.

The following lemma implies that pulling-back the connection form or the weak covariant derivative yields the same result. 
\begin{lemma}\label{lemma: pull-back of connection}
	Let \(W \subset \R^{n+1}_+\) be an open and bounded domain and let \(\psi\colon \overline W \to \Omega\) be a diffeomorphism. 
	Then
	\begin{equation*}
		U\circ \psi \in W^{1,p}_{\psi^\ast \Gamma }(W,F)\iff  U \in  W^{1,p}_\Gamma(\Omega,F)
	\end{equation*}
	and for almost every \(x \in W\)
	\begin{equation*}
		\Diff_{\psi^\ast \Gamma }(U \circ \psi)\brk{x} = (\psi^\ast (\Diff_\Gamma  U))\brk{x}.
	\end{equation*}
	Consequently, for almost every \(x\in W\),
	\begin{equation*}
		\frac{|\Diff_\Gamma  U\brk{\psi\brk{x}} |}{\|\Diff (\psi^{-1})\|_{\infty}} \leq |\Diff_{\psi^\ast \Gamma }(U\circ \psi)\brk{x} | \leq \| \Diff\psi\|_\infty |\Diff_\Gamma  U\brk{\psi\brk{x}} |.
	\end{equation*}
\end{lemma}
\begin{proof}
	By the chain rule, we have for almost every \(x \in W\)
	\begin{equation}\label{eq: pullback1}
	\begin{split}
		\Diff_{\psi^\ast \Gamma }(U\circ \psi)\brk{x} 
		&= \Diff U\brk{\psi\brk{x}} + \psi^\ast\Gamma \brk{x} U\brk{\psi\brk{x}}\\
		&=\Diff_\Gamma  U \brk{\psi\brk{x}} \circ \Diff \psi\brk{x} = \psi^\ast (\Diff_\Gamma  U)\brk{x}.
	\end{split}
	\end{equation}
	Since \(\psi\) is a diffeomorphism up to the boundary, it follows that 
	\begin{equation} \label{eq: pullback2}
		| \Diff_{\psi^\ast \Gamma }(U\circ \psi)\brk{x} | \leq \|\Diff \psi\|_\infty| \Diff_\Gamma  U (\psi\brk{x} ) |
	\end{equation}
	and 
	\begin{equation}\label{eq: pullback3}
	\begin{split}
		|\Diff_\Gamma  U(\psi\brk{x} ) | 
		&= |\Diff_{\psi^\ast \Gamma } (U\circ \psi)\brk{x} \circ \Diff (\psi^{-1})\brk{\psi\brk{x}} | \\
		&\leq \| \Diff (\psi^{-1})\|_\infty|\Diff_{\psi^\ast \Gamma }(U \circ \psi)\brk{x} |.
	\end{split}
	\end{equation}
	The conclusion follows from \eqref{eq: pullback1}, \eqref{eq: pullback2} and \eqref{eq: pullback3}.
\end{proof}

In order to define the trace space, we consider a bundle \(\mathcal E\) over \(\partial \mathcal{M}\) and a diffeomorphism on its image \(\psi \colon V \to \partial \mathcal{M}\) such that \(\mathcal E\) is trivial above \(\psi(V)\).
Then, given a continuous connection \(K_\shortparallel\) on \(\mathcal E\) which is locally encoded in \(\psi(V)\) by the connection form \(\Gamma_\shortparallel\) and given \(\gamma \in C^1 \brk{[0, 1], \psi\brk{V}}\) we can consider the parallel transport \(\Pt_{\tilde\gamma}^{\psi^* \Gamma_\shortparallel} \brk{t}\) along the path \(\tilde\gamma = \psi^{-1} \circ \gamma\) defined by \eqref{def: Pt}. It turns out that this operator does not depend on \(\psi\).
\begin{lemma}
	\label{lemma: Pt on manifold is well-defined}
	Let \(\Omega \subset \partial\mathcal{M}\) be a local trivialisation with a continuous connection form \(\Gamma_\shortparallel\). If \(\psi_i \colon V_i \to \Omega\), \(i= 1\), \(2\),  are two diffeomorphisms and \(\gamma \in C^1\brk{\sqb{0,1}, \Omega}\), then, for all \(t\in \sqb{0,1}\),
	\begin{equation*}
		\Pt^{\psi_1^\ast \Gamma_\shortparallel}_{{\tilde\gamma}_1}(t) = \Pt^{\psi_2^\ast \Gamma_\shortparallel}_{\tilde{\gamma}_2}(t),
	\end{equation*}
	where \({\tilde \gamma}_i = \psi_i^{-1} \circ \gamma\), \(i=1\), \(2\).
\end{lemma}
\begin{proof}
	It suffices to observe that both operators satisfy the same final value problem, and we conclude by uniqueness of the solution. This follows from the definition \eqref{def: Pt} of parallel transport as well as the definition of pull-back \eqref{def: pullback Gamma}.
\end{proof}
In view of \cref{lemma: Pt gauge-invariance,lemma: Pt on manifold is well-defined}, \(\Pt_{\tilde\gamma}^{\psi^* \Gamma_{\shortparallel}} \brk{t}\) induces a well-defined operator from \(\mathcal{E}_{\gamma \brk{1}}\) to \(\mathcal{E}_{\gamma\brk{t}}\) that only depends on the connection \(K_\shortparallel\) on \(\partial \mathcal{M}\). We denote this operator by \(\Pt^{K_\shortparallel}_\gamma \brk{t}\). (Even though it will not be need in the sequel, it can be noted that \(\Pt^{K_\shortparallel}_\gamma \brk{t}\) can then be defined beyond a local trivialisation.)

Since \(\partial\mathcal{M}\) is a compact and smooth manifold by assumption, it has a positive injectivity radius \(\mathrm{inj}_{\partial \mathcal{M}}>0\), and if \(x\), \(y\in\partial\mathcal{M}\) are such that \(\operatorname{dist}_{\partial N}\brk{x, y}<\mathrm{inj}_{\partial \mathcal{M}}\), then there exists a unique distance minimising geodesic
\(
	\zeta_{x,y}\colon [0,1] \to \partial\mathcal{M}
\)
with \(\zeta_{x,y}(0)=x\) and \(\zeta_{x,y}(1)=y\), see~\cite{doCarmo}*{Ch.\ 13}.
We then define in the spirit of \eqref{eq: def R} 
\begin{equation}
	\label{eq: def R on manifold}
  R_{\partial\mathcal{M}}^{K_\shortparallel} \brk{x, y}
  = \Pt^{K_\shortparallel}_{\zeta_{x,y}}(0),
\end{equation}
the parallel transport from \(y\) to \(x\) along \(\zeta_{x,y}\).
\begin{definition}
	Let \(n\geq 1\), let \(1 \leq p < +\infty\), let \(0<s<1\) and let \(\mathcal{M}\) be an \((n+1)\)-dimensional smooth and compact manifold with boundary \(\partial\mathcal{M}\). Let \(\mathcal{E}\) be a smooth Riemannian vector bundle over \(\partial \mathcal{M}\) and let \(K_\shortparallel\) be a continuous metric connection on \(\mathcal{E}\). We define the \emph{fractional gauge-covariant Sobolev space} \(W^{s,p}_{K_\shortparallel}(\partial \mathcal{M}, \mathcal{E})\) as the set of all sections \(u\in L^p(\partial \mathcal{M}, \mathcal{E})\) such that the \emph{gauge-covariant Gagliardo seminorm}
	\begin{equation}\label{gauge cov Gagliardo_19572}
		\norm{u}_{W^{s,p}_{K_\shortparallel}(\partial\mathcal{M}, \mathcal{E})}^p = \smashoperator[r]{\iint_{\substack{\brk{x, y}\in\partial \mathcal{M} \times \partial \mathcal{M}\\ \operatorname{dist}_{\partial \mathcal{M}} \brk{x, y} < \mathrm{inj}_{\partial \mathcal{M}}}}} \frac{\norm{u(x)-R^{K_\shortparallel}_{\partial\mathcal{M}}\brk{x, y}u\brk{y}}^p}{\operatorname{dist}_{\partial \mathcal{M}}\brk{x, y}^{n+sp}}\dd x \dd y
	\end{equation}
	is finite. 
\end{definition}
We endow the space \(W^{s,p}_{K_\shortparallel}(\partial\mathcal{M}, \mathcal{E})\) with the norm
\begin{equation*}
	\|u\|_{W^{s,p}_{K_\shortparallel}(\partial\mathcal{M}, \mathcal{E})} = \brk[\Big]{\Norm{u}_{L^p(\partial\mathcal{M}, \mathcal E)}^p + \norm{u}^p_{W^{s,p}_{K_\shortparallel}(\partial\mathcal{M}, \mathcal{E})}}^{1/p}.
\end{equation*}
A standard proof shows that this defines a Banach space.

On the other hand, if \(\psi \colon V \to \partial \mathcal{M}\) with \(V \subseteq \R^n\) convex, we can consider for \(x, y \in V\)
\(
 R^{\psi^* \Gamma_\shortparallel } (x, y)
\)
defined in \eqref{eq: def R}.
The corresponding gauge-covariant Gagliardo seminorms with charts \(\psi\colon V \to \partial \mathcal{M}\) of the boundary are related in a non-local counterpart of the equivariance under pull-back of covariant derivatives of \cref{lemma: pull-back of connection}.

\begin{lemma}\label{lemma 6.2 in article}
	Let \(n\geq 1\), \(1\leq p < +\infty\) and \(0 < s < 1\). Let \(V \subset \R^n\) be open and bounded and let \(\psi \colon V \to \partial\mathcal{M}\) be a diffeomorphism up to the boundary on its image such that \(\operatorname{diam}(\psi(V)) < \mathrm{inj}_{\partial \mathcal{M}}\).
	If \(\Gamma_\shortparallel \in   C^1(\partial\mathcal{M}, \Lin{T(\partial\mathcal{M})}{\mathfrak{o}(F)})\), if \(V\) is convex and if \(\psi(V)\) is geodesically convex, then there exists a positive constant \(C\) such that for every \(u \in L^p(\psi(V), F)\),
	\begin{multline*}
		\smashoperator{\iint_{\psi(V)\times \psi(V)}} \frac{\norm{u\brk{x} - R^{\Gamma_\shortparallel}_{\partial\mathcal{M}}\brk{x, y}u\brk{y}}^p}{\operatorname{dist}_{\partial\mathcal{M}}\brk{x, y}^{n+sp}}\dd  y \dd  x \\
		\leq C\brk[\bigg]{\smashoperator[r]{\iint_{V\times V}} \frac{\norm{u(\psi(v))  - R^{\psi^\ast \Gamma_\shortparallel}(v,w)u(\psi(w))}^p}{|w-v|^{n+sp}} \dd w \dd v \\
		+ \min\{ \| \mathscr K_{\Gamma_\shortparallel}\|^p_\infty, \|\mathscr K_{\Gamma_\shortparallel}\|_\infty^{sp/3} \} \int_V | u(\psi(v))|^p \dd v}
	\end{multline*}
	and 
	\begin{multline*}
		\smashoperator{\iint_{V\times V}} \frac{|u(\psi(v))  - R^{\psi^\ast \Gamma_\shortparallel}(v,w)u(\psi(w))|^p}{|w-v|^{n+sp}} \dd w \dd v \\
		\leq C \brk[\bigg]{\smashoperator[r]{\iint_{\psi(V) \times \psi(V)}} \frac{|R^{\Gamma_\shortparallel}_{\partial\mathcal{M}}\brk{x, y}u\brk{y} - u\brk{x} |^p}{\operatorname{dist}_{\partial\mathcal{M}}\brk{x, y}^{n+sp}}\dd  y \dd  x   \\
		+ \min\{ \| \mathscr K_{\Gamma_\shortparallel}\|^p_\infty, \|\mathscr K_{\Gamma_\shortparallel}\|_\infty^{sp/3} \} \int_{\psi(V)} | u\brk{x}|^p \dd  x }.
	\end{multline*}
\end{lemma}
The key tool to prove \cref{lemma 6.2 in article} is the following estimate comparing the two parallel transport operators.
\begin{lemma}\label{comparaison Aprime et psi* Aprime}
	Let \(n\geq 1\), let \(V\subset \R^n\) be open and bounded and let \(\psi \colon V \to \partial \mathcal{M}\) be a diffeomorphism up to the boundary on its image such that \(\operatorname{diam}(\psi(V)) < \mathrm{inj}_{\partial \mathcal{M}}\).
	If \(\Gamma_\shortparallel \in  C^1(\partial\mathcal{M},\Lin{\R^n}{\mathfrak{o}(F)})\), if \(V\) is convex and if \(\psi(V)\) is geodesically convex, then there exists a positive constant \(C > 0\) such that for all \(x\), \(y\in V\)
	\begin{equation*}
	|R^{\psi^\ast \Gamma_\shortparallel}\brk{x, y}R^{\Gamma_\shortparallel}_{\partial \mathcal{M}}(\psi\brk{y},\psi\brk{x}) - \id_F| \leq C \min\{2, \| \mathscr K_{\Gamma_\shortparallel}\|_{L^\infty(\psi(V))} \norm{x - y}^3 \}.
	\end{equation*}
\end{lemma}

\begin{proof}
Let \(\zeta\) be the minimal geodesic in \(\partial \mathcal{M}\) joining \(\psi \brk{x}\) to \(\psi \brk{y}\).
We define 
	\begin{equation}
	 \gamma_s \brk{t} = \brk{1 - s} \brk{(1-t)x+ty} + s {\psi^{-1}}\brk{\zeta (t)}.
	\end{equation}
	Since \(\zeta\) is a geodesic, \(\gamma_1 = \psi^{-1} \circ \zeta\) 
	satisfies the local geodesic equation
	\begin{equation*}
		\gamma_1^{\prime\prime}(t) = \Xi(\gamma_1(t)) [\gamma_1^\prime(t), \gamma_1^\prime(t)]
	\end{equation*}
	where \(\Xi\) is a symmetric \(2\)-form (defined with the Christofell symbols of coming from the Riemannian metric \(\mathcal{M}\) and the local chart \(\psi\)), see for example \cite{doCarmo}*{Ch.\ 3}.
	Since geodesics have constant speed proportional to the Euclidean distance between their endpoints, it follows that there exists \(C>0\) such that 
	\begin{equation}\label{zeta primeprime}
		\norm{\gamma_1^{\prime\prime}\brk{t}} \leq C \norm{x - y}^2.
	\end{equation}
	By the fundamental theorem of calculus and integration by parts, we obtain 
	\[
	\begin{split}
	 \gamma_1 \brk{t} - \gamma_0 \brk{t}
	 &= \gamma_1 \brk{t} - \brk{(1-t)x+ty}\\
	 &= \brk{1 - t} \int_{0}^t \gamma_1' \brk{\tau} \dd \tau - 
	 t \int_t^1 \gamma_1' \brk{\tau} \dd \tau\\
	 &= -\brk[\Big]{\brk{1 - t}\int_{0}^t \tau \gamma_1'' \brk{\tau} \dd \tau
	 + t \int_t^1 \brk{1 - \tau} \gamma_1'' \brk{\tau} \dd \tau},
	\end{split}
	\]
	and therefore, using \eqref{zeta primeprime}, we have 
	\[
	\begin{split}
	  \norm[\big]{\tfrac{\dd \gamma_s}{\dd s} \brk{t}}& \leq \brk[\Big]{(1-t)\int_0^t \tau \dd \tau + t\int_t^1 (1-\tau) \dd \tau} \Norm{\gamma_1^{\prime\prime}}_{L^\infty} \\
	  &= \tfrac{t \brk{1 -t}}{2} \Norm{\gamma_1''}_{L^\infty}
	  \le C \norm{x - y}^2.
	\end{split}
	\]
	We note that \(\gamma_0 = \gamma_{x, y}\) whereas \(\gamma_1 = \psi^{-1}\circ\zeta\), \(\gamma_s \brk{0} = x\) and \(\gamma_s \brk{1} = y\).
	By \cref{lemma: formula for derivative of Pt wrt parameter}, we have
	\begin{equation*}
\norm[\Big]{\tfrac{\dd}{\dd s}\Pt^\Gamma_{\gamma_s}\brk{0}} \le \int_0^1 \norm{\mathscr{K}_\Gamma(\gamma_s(\tau))[\gamma_s^\prime(\tau), \tfrac{\dd}{\dd s} \gamma_s(\tau)]} \dd \tau
\le  C \|\mathscr K_{\Gamma_\shortparallel} \|_{L^\infty(\psi(V))} \norm{x - y}^3,
\end{equation*}
and the conclusion follows.
\end{proof}
We are now ready to characterise traces and extensions of gauge-covariant Sobolev spaces for a Riemannian vector bundle \(\mathcal{E}\) which has a smooth and compact Riemannian manifold \(\mathcal{M}\) with boundary \(\partial \mathcal{M}\) and interior \(\mathcal{M}\) as base space and is endowed with a \(C^1\) connection \(K\). Under this assumption on \(K\), and thus the connection form \(\Gamma\) encoding \(K\) on a trivialisation \(\Omega\subset \mathcal{M}\), the gauge-covariant Sobolev space \(W^{1,p}_K(\mathcal{M}, \mathcal{E})\) is the same as the usual vector valued Sobolev space \(W^{1,p}(\mathcal{M},F)\) since the connection form \(\Gamma\) is bounded on any trivialisation. Therefore, the usual trace theorem applies and yields a linear and continuous trace operator
\begin{equation*}
	\operatorname{Tr}\colon  W^{1,p}_K(\mathcal{M}, \mathcal{E}) \to W^{1-1/p}(\partial\mathcal{M}, F).
\end{equation*}
However, the norm estimates are not gauge-invariant, and the trace space should depend on the connection \(K\) in a gauge-invariant manner. The results that follow achieve this enhanced characterisation.

We obtain the following estimates on the traces of gauge-covariant Sobolev spaces, generalising the particular case of \cite{Nguyen_VanSchaftingen_2020}*{Prop.\ 6.4} regarding magnetic Sobolev spaces. In this and the following proof, we denote by \(C\) positive constants independent of \(\Gamma \), \(u\) or \(U\). 
\begin{proposition}\label{prop: trace sur domaine}
	Let \(n\geq 1\), let \(1\leq p < +\infty\) and let \(\mathcal{M}\) be a smooth and compact Riemannian manifold of dimension \(n+1\) with boundary \(\partial\mathcal{M}\). Let \(\mathcal{E}\) be a Riemannian vector bundle with model fibre \(F\) and base space \(\mathcal{M}\). If \(\mathcal{E}\) is endowed with a metric \(C^1\) connection \(K\) with curvature \(\mathscr K_K\) satisfying \(\Norm{\mathscr K_K}_\infty \leq \beta\), then there exists a constant \(C=C(\mathcal{M},p)>0\) depending only on \(\mathcal{M}\) and \(p\) such that for all \(U\in  W^{1,p}_K(\mathcal{M}, \mathcal{E})\) we have, with \(u = \operatorname{Tr}U\),
	\begin{equation*}
		\smashoperator[r]{\iint_{\substack{\brk{x, y}\in\partial\mathcal{M}\times\partial\mathcal{M} \\ \operatorname{dist}_{\partial \mathcal{M}}\brk{x, y}<\mathrm{inj}_{\partial \mathcal{M}}}}}
		\frac{|u\brk{x} -R^{K_\shortparallel}_{\partial\mathcal{M}}\brk{x, y} u\brk{y}|^p}{\operatorname{dist}_{\partial\mathcal{M}}(x, y)^{n+p-1}} \dd  x \dd  y
		\leq C \int_{\mathcal{M}} |\Diff_K  U\brk{x} |^p + (1+\beta^{p/2})\norm{U \brk{x}}^p \dd x.
	\end{equation*}
\end{proposition}
Here and in what follows, we define on the boundary \(\partial\mathcal{M}\) the \(C^1\) connection \(K_\shortparallel\) locally encoded by the connection form \(\Gamma_\shortparallel\) as the restriction of \(\Gamma\) to \(\pi_{\mathcal{E}}^{-1} \brk{\partial \mathcal{M}}\), or equivalenty, as the pull-back by the inclusion map from \(\partial \mathcal{M}\) to \(\mathcal{M}\).

\begin{proof}[Proof of \cref{prop: trace sur domaine}]
	Since smooth sections \(C^\infty(\mathcal{M}, \mathcal{E})\) are dense in \(W^{1,p}(\mathcal{M},F)=  W^{1,p}_K(\mathcal{M}, \mathcal{E})\), see for example~\cite{brezis}*{Cor.\ 9.8}, we may assume \(U\in  C^\infty(\mathcal{M}, F)\). By the divergence theorem and \cref{lemma: diamagnetic inequality}, we have
	\begin{equation}\label{extensio domaine eq1}
		\int_{\partial \Omega} \norm{U \brk{x}}^p \dd  x
		\leq C \brk[\bigg]{\int_{\Omega} |\Diff_\Gamma  U\brk{x} |^p + \norm{U \brk{x}}^p \dd x }^{1/p}\brk[\bigg]{\int_{\Omega} \norm{U \brk{x}}^p \dd x }^{1-1/p}
	\end{equation}
	for any local trivialisation \(\Omega \subset \mathcal{M}\). Since \(\partial\mathcal{M}\) is smooth and compact, there exists a finite number of maps, say \(j\in\{1, \dotsc, k\}\),
	\begin{equation*}
		\psi_j\colon B_1 \times [0,1) \to \mathcal{M}
	\end{equation*}
	that are diffeomorphisms on their image up to the boundary such that \(\mathcal{E}\) is trivial over \(\psi_j(B_1 \times [0,1))\), 
	\begin{gather*}
		\psi_j( B_1 \times \{0\}) = \psi_j(B_1 \times [0,1))\cap \partial\mathcal{M}, \\
		\partial\mathcal{M} \subset \bigcup_{j=1}^k \psi_j(B_{1/2}\times \{0\})
	\end{gather*}
	and, for every \(j\in\{1, \dotsc , k\}\), \(\psi_j(B_1 \times \{0\})\) is geodesically convex with 
	\begin{equation*}
		\operatorname{diam}(\psi_j(B_1 \times \{0\} ))< \mathrm{inj}_{\partial \mathcal{M}}.
	\end{equation*}
	Since the vector bundle \(\mathcal{E}\) is trivial over \(\psi_j(B_1 \times [0,1))\), there exists a connection form \(\Gamma_j \in C^1\brk[\big]{\psi_j(B_1 \times [0,1)), \Lin{T\mathcal{M}}{\mathfrak{o}(F)}}\) such that \(\Diff_K = \Diff_{\Gamma_j}\) on \(\psi_j(B_1 \times [0,1))\).
	In virtue of \cref{prop: trace estimates local version} with \(s=1-\frac{1}{p}\), we have, for \(j\in\{1,\dotsc, k\}\),
	\begin{multline}\label{opqr}
		\int_{B_1}\int_{B_1} \frac{|u (\psi_j(x,0) )  - R^{\psi_j^\ast \Gamma_{j\shortparallel}}\brk{x, y} u(\psi_j(y,0) )|^p}{|x-y|^{n+p-1}} \dd x \dd y \\
		\leq C \iint_{B_1 \times (0,1)}|\Diff_{\psi_j^\ast \Gamma_{j\shortparallel}}(U\circ \psi_j)(z)  |^p + \beta^{p/2} |(U \circ \psi_j)(z)|^p \dd z.
	\end{multline}
	Since \(\beta\) is an arbitrary upper bound on \(\|\mathscr K_{K}\|_\infty\), we may assume without loss of generality that \(\beta \geq 1\). \Cref{lemma: pull-back of connection} and \cref{lemma 6.2 in article} then imply together with \eqref{opqr} that 
	\begin{multline}\label{opqr2}
		\smashoperator[r]{\iint_{\psi_j(B_1\times\{0\}) \times \psi_j(B_1\times \{0\})}}
		\frac{|u\brk{x} - R^{\Gamma_{j\shortparallel}}_{\partial\mathcal{M}}\brk{x, y} u\brk{y}|^p}{\operatorname{dist}_{\partial\mathcal{M}}\brk{x, y}^{n+p-1}} \dd  x \dd  y \\
		\leq C \brk[\bigg]{\smashoperator[r]{\int_{\psi_j(B_1\times(0,1)) }} |\Diff_{\Gamma_j}  U\brk{x} |^p + \beta^{p/2}| U\brk{x} |^p  \dd x + \beta^{(p-1)/3} \smashoperator{\int_{\psi_j(B_1 \times \{0\}) }} \norm{u\brk{x}}^p \dd  x }.
	\end{multline}
	Since \(\beta^{(p-1)/3} \leq \beta^{(p-1)/2}\), using \eqref{extensio domaine eq1} and	 Young's inequality we obtain 
	\begin{equation}
	\label{opqr3}
		\beta^{\frac{p-1}{3}} \smashoperator{\int_{\psi_j(B_1\times\{0\})}} \norm{u\brk{x}}^p \dd  x
	\leq C \smashoperator{\int_{\psi_j(B_1\times (0,1))}} \norm{\Diff_{\Gamma_j}  U\brk{x}}^p + (1+\beta^{p/2}) \norm{U\brk{x}}^p \dd x.
	\end{equation}
	The conclusion follows by summing over \(j\in \{1, \dotsc , k\}\) and combining the estimate \eqref{opqr2} with the estimate \eqref{opqr3}.
\end{proof}
Regarding extensions, we have the following result generalising its analogue for magnetic Sobolev spaces~\cite{Nguyen_VanSchaftingen_2020}*{Prop.\ 6.5}.
\begin{proposition}\label{prop extensions sur domaine}
	Let \(n\geq 1\), let \(1\leq p < +\infty\) and let \(\mathcal{M}\) be a smooth and compact Riemannian manifold of dimension \(n+1\) with boundary \(\partial\mathcal{M}\) and interior \(\mathcal{M}\). Let \(\mathcal{E}\) be a Riemannian vector bundle with model fibre \(F\) and base space \(\mathcal{M}\). If \(\mathcal{E}\) is endowed with a metric \(C^1\) connection \(K\) with curvature \(\mathscr K_K\) satisfying \(\Norm{\mathscr K_K}_\infty \leq \beta\), then there exists a constant \(C=C(\mathcal{M},p)>0\) depending only on \(\mathcal{M}\) and \(p\) such that for all \(u\in W_{K_\shortparallel}^{1-1/p,p}(\partial\mathcal{M}, \mathcal{E})\), there exists \(U\in  W^{1,p}_K(\mathcal{M}, \mathcal{E})\cap   C^1(\mathcal{M}, \mathcal{E})\) such that \(\operatorname{Tr}U=u\),
	\begin{multline*}
		\int_{\mathcal{M}} |\Diff U\brk{x} |^p \dd x
		\leq C \biggl( \iint\displaylimits_{\substack{\brk{x, y}\in \partial\mathcal{M} \times \partial\mathcal{M} \\ \operatorname{dist}_{\partial \mathcal{M}}\brk{x, y}<\mathrm{inj}_{\partial \mathcal{M}}}} \frac{|u\brk{x} - R^{K_\shortparallel}_{\partial\mathcal{M}}\brk{x, y} u\brk{y} |^p}{\operatorname{dist}_{\partial\mathcal{M}}\brk{x, y}^{n+p-1}} \dd y \dd  x \\[-1em]
		+ (1+\beta^{(p-1)/2}) \int_{\partial\mathcal{M}} \norm{u\brk{x}}^p \dd x \biggr)
	\end{multline*}
	and 
	\begin{equation*}
		\int_{\mathcal{M}} | U\brk{x} |^p \dd x \leq \frac{C}{1+\beta^{p/2}} \int_{\partial\mathcal{M}} \norm{u\brk{x}}^p \dd x.
	\end{equation*}
	\begin{proof}
		Since \(\partial\mathcal{M}\) is smooth and compact, there exists a finite number of maps, say \(j\in\{1, \dotsc, k\}\),
		\begin{equation*}
			\psi_j\colon B_1 \times [0,1) \to \mathcal{M}
		\end{equation*}
		that are diffeomorphisms on their image up to the boundary such that \(\mathcal{E}\) is trivial over \(\psi_j(B_1\times[0,1))\), 
		\begin{gather*}
			\psi_j( B_1 \times \{0\}) = \psi(B_1 \times [0,1))\cap \partial\mathcal{M}, \\
			\partial\mathcal{M} \subset \bigcup_{j=1}^k \psi_j(B_{1/2}\times \{0\})
		\end{gather*}
		and, for every \(j\in\{1, \dotsc , k\}\), \(\psi_j(B_1 \times \{0\})\) is geodesically convex with 
		\begin{equation*}
			\operatorname{diam}(\psi(B_1 \times \{0\} )) < \mathrm{inj}_{\partial \mathcal{M}}.
		\end{equation*}
		Moreover, there exists a collection of smooth functions \(\{\eta_j\}_{j=1}^k \subset   C^\infty(\overline{\mathcal{M}},\R)\) such that for all \(j\in \{1,\dotsc, k\}\)
		\begin{equation*}
			\Supp \eta_j \subset \psi_j(B_{1/2}\times [0,1/2]) 
		\end{equation*}
		and \(\sum_{j=1}^k \eta_j = 1\) on \(\mathcal{M}\).
		
		In virtue of \cref{prop: extension local version} with \(s=1-\frac{1}{p}\), for every \(j\in\{1,\dotsc ,k\}\), there exists 
		\begin{equation*}
			U_j \in W^{1,p}_{\psi_j^\ast \Gamma_j }(B_{1/2} \cap [0,1/2], \mathcal{E}) \cap   C^1 ( B_{1/2} \times [0,1/2], \mathcal{E})    
		\end{equation*}
		such that \(\operatorname{Tr}U_j = u \circ \psi_j\) on \(B_{1/2}\times\{0\}\) and furthermore \(U_j\) satisfies
		\begin{multline}\label{A eq1}
			\int_{B_{1/2}}\int_0^{\frac12} |\Diff_{\psi_j^\ast \Gamma_j } U_j(x,t) |^p \dd t \dd x \\
			\leq C \biggl( \int_{B_1}\int_{B_1} \frac{|u(\psi_j(x,0))  - R^{\psi_j^\ast \Gamma_{j\shortparallel}}\brk{x, y} u(\psi_j(y,0))|^p}{\norm{x - y}^{n+p-1}} \dd x \dd y \\
			+ \beta^{(p-1)/2} \int_{B_1} |u(\psi_j(x,0)) |^p  \dd x\biggr)
		\end{multline}
		and, assuming without loss of generality that \(\beta \geq 1\),
		\begin{equation}\label{A eq2}
			\int_{B_{1/2}}\int_0^{\frac{1}{2}} |\Diff_{\psi_j^\ast \Gamma_{j\shortparallel}}U_j\brk{x}|^p \dd x\leq C \int_{B_1} |(u \circ \psi)(x,0) |^p \dd x. 
		\end{equation}
		Applying \cref{lemma: pull-back of connection} and \cref{lemma 6.2 in article} to \eqref{A eq1} and \eqref{A eq2}, we obtain
		\begin{multline}\label{A3}
			\smashoperator[r]{\int_{\psi_j( B_{1/2}\times[0,1/2])}} | \Diff_{\Gamma_j} (U_j\circ \psi_j^{-1})\brk{x} |^p \dd x \\
			\leq C \biggl( \smashoperator[r]{\iint_{\psi_j(B_1 \times \{0\}) \times \psi_j(B_1 \times\{0\})}} \frac{\norm{u\brk{x} - R^{\Gamma_{j\shortparallel}}_{\partial\mathcal{M}}\brk{x, y}u\brk{y}}^p }{\operatorname{dist}_{\partial\mathcal{M}}\brk{x, y}^{n+p-1}} \dd x\dd y + \beta^{(p-1)/2} \smashoperator{\int_{\psi_j(B_1 \times\{0\})}} \norm{u\brk{x}} \dd x \biggr)
		\end{multline}
		and
		\begin{equation}\label{A4}
			\smashoperator{\int_{\psi_j( B_{1/2}\times [0,1/2])}} | U_j\brk{x} |^p \dd x \leq C \smashoperator[r]{\int_{\psi_j(B_1\times\{0\})}} \norm{u\brk{x}}^p \dd x.
		\end{equation}
		We define on \(\mathcal{M}\)
		\begin{equation*}
			U = \sum_{j=1}^k \eta_j (U_j \circ \psi_j^{-1}).
		\end{equation*}
		By the Leibniz rule for covariant derivatives we have
		\begin{equation}\label{A5}
			\Diff_K  U = \sum_{j=1}^k ( (U_j \circ \psi^{-1} )\Diff \eta_j + \eta_j \Diff_K (U_j \circ \psi_j^{-1}) ).
		\end{equation}
		The conclusion follows from \eqref{A3}, \eqref{A4} and \eqref{A5}.
	\end{proof}
\end{proposition}

\end{document}